\documentclass[11pt,a4paper]{amsart}

\usepackage[T1]{fontenc}
\usepackage{amssymb,graphics,graphicx}
\usepackage{enumitem}

 \newcommand\gla{{\bf \lambda}} 
\newcommand\glaD{{\bf \lambda}^D} 
\newcommand\ep{\varepsilon} 

\newcommand\mk{\medskip}
\newcommand{\R}{\mathbb{R}} %
\newcommand{\Z}{\mathbb{Z}} %
\newcommand{\N}{\mathbb{N}} %
\newcommand{\T}{{\zu}} %
\newcommand{\suchthat}{: } %
\newcommand{\accol}[1]{{\left\{ #1 \right\}}} %
\newcommand{\biggaccol}[1]{{\biggl\{ #1 \biggr\}}} %
\newcommand{\set}[2]{\accol{#1 \suchthat #2}} %
\newcommand{\biggset}[2]{\biggaccol{#1 \suchthat #2}} %
\newcommand{\abs}[1]{{\left\lvert #1 \right\rvert}} %
\newcommand{\biggabs}[1]{{\biggl\lvert #1 \biggr\rvert}} %
\newcommand{\smallabs}[1]{{| #1 |}} %
\newcommand{\paren}[1]{{\left( #1 \right)}} %
\newcommand{\biggparen}[1]{{\biggl( #1 \biggr)}} %
\newcommand{\floor}[1]{\left\lfloor#1\right\rfloor} 
\newcommand{\norm}[2]{{\left\|#2\right\|}_{#1}} %
\newcommand{\deq}{\mathrel{\mathop:} = } %
\newcommand{\eps}{\varepsilon} %
\newcommand\zu{[0,1]}
\DeclareMathOperator{\loc}{loc}

\DeclareMathOperator{\Gr}{Gr}

\newcommand\bfk{{\bf k}}
\newcommand\bfK{{\bf K}}
\newcommand\bfl{{\bf l}}

\newtheorem{theorem}{Theorem}[section] %
\newtheorem{proposition}[theorem]{Proposition} %
\newtheorem{corollary}[theorem]{Corollary} %
\newtheorem{lemma}[theorem]{Lemma} %
\newtheorem{definition}{Definition} %
\newtheorem{remark}[theorem]{Remark} %
\newtheorem{conjecture}[theorem]{Conjecture} %

\newcommand{\Lebesgue}{\mathcal{L}}
\newcommand{\Borel}{\mathcal{B}}
\newcommand{\Fourier}{\mathcal{F}}

\newcommand{\calA}{\mathcal{A}}
\newcommand{\calF}{\mathcal{F}}
\newcommand{\calG}{\mathcal{G}}
\newcommand{\calH}{\mathcal{H}}
\newcommand{\calP}{\mathcal{P}}
\newcommand{\calB}{\mathcal{B}}

\newcommand{\calT}{\mathcal{T}}
\newcommand{\calX}{\mathcal{X}}
\newcommand{\calY}{\mathcal{Y}}
\newcommand{\calK}{\mathcal{K}}
\newcommand{\calO}{\mathcal{O}}

\newcommand{\Aun}{\calA_1}

\newcommand{\Besov}[3]{B^{#1}_{#2, #3}}

\begin{document}

\title[Local behavior of traces of Besov functions]{Local behavior of traces of Besov functions: \\ Prevalent results}

\author[J.-M.  Aubry ]{Jean-Marie Aubry}

\address{Laboratoire d'Analyse et Math\'ematiques Appliqu\'ees (CNRS UMR
  8050) \\
  Universit\'e Paris-Est- Cr\'eteil - Val-de-Marne \\ 61 av. du G\'en\'eral de Gaulle \\
  94010 Cr\'eteil  \\
  France}

\email{jmaubry@math.cnrs.fr}

\author[D.  Maman]{Delphine Maman}

\email{delphine.maman@univ-paris12.fr}

\author[S.  Seuret]{St\'ephane Seuret}

\email{seuret@univ-paris12.fr}



\begin{abstract}
  Let $1 \leq d < D$ and $(p, q, s) $ satisfying $0 < p < \infty$, $0
  < q \leq \infty$, $0 < s- d/p < \infty$.  In this article we study
  the global and local regularity properties of traces, on affine subsets of
  $\R^D$, of functions belonging to the Besov space
  $\Besov{s}{p}{q}(\R^D)$.  Given a $d$-dimensional subspace
  $\mathcal{H} \subset \R^D$, for almost all functions in
  $\Besov{s}{p}{q}(\R^D)$ (in the sense of prevalence), we are able to
  compute the singularity spectrum of the traces $f_a$ of $f$ on
  affine subspaces of the form $a+\mathcal{H}$, for Lebesgue-almost
  every $a \in \R^{D-d}$.  In particular, we prove that for
  Lebesgue-almost every $a \in \R^{D - d}$, these traces $f_a$ are
  more regular than what could be expected from standard trace
  theorems, and that $f_a$ enjoys a multifractal behavior.
\end{abstract}

\maketitle

\section{Introduction}
\label{sec:introduction}

%

Investigating regularity properties of traces of functions belonging
to some Besov or Sobolev spaces is a longstanding issue. For instance,
such questions arise from PDE's theory, where the Dirichlet condition
imposes some regularity properties of the trace of the solution on the
frontier of the domain.  In this article, we study the local behavior
of traces of functions belonging to the Besov space
$\Besov{s}{p}{q}(\R^D)$ on $d$-dimensional affine subspaces of $\R^D$.

\medskip

Not only concerned with global smoothness properties (i.e. to which
Sobolev and Besov spaces the traces belong), we will especially focus
on the local behavior of such traces.  The notion of pointwise
regularity we discuss in the sequel is the following.  Given a real
function $f \in L^\infty_{\loc}(\R^D)$ and $x_0 \in \R^D$, $f$ is said to
belong to ${\mathcal C}^{\alpha}(x_0)$, for some $\alpha\geq 0$, if
there exists a polynomial $P$ of degree at most $\lfloor\alpha\rfloor$
and a constant $C>0$ such that locally around $x_0$~:
\begin{equation}
  \label{defpoint}
  |f(x)-P(x-x_0)|\leq C|x-x_0|^{\alpha}.
\end{equation}
The local regularity of $f$ at $x_0$ is measured by the {\it pointwise
  H\"older exponent}~:
$$h_f(x_0)=\sup\{\alpha\geq 0:~f \in {\mathcal C}^{\alpha}(x_0)\}.$$
As will be observed soon, this exponent $h_f(x_0)$ may vary rather
erratically with $x_0$, and the relevant information is then provided
by the {\it spectrum of singularities} $d_f$ of $f$, which is the
function~: $$d_f: h \in[0,\infty] \longmapsto \mbox{dim}_{{\calH}}
\ E_f(h), \ \ \mbox{ where } E_f(h):= \{x_0\in \R^D: h_f(x_0)=h\}.$$
Here $\mbox{dim}_{{\calH}}$ stands for the Hausdorff dimension. We
adopt the convention that $\dim_{{\calH}} \, \emptyset =-\infty$. The
spectrum of singularities $d_f$ describes the geometrical repartition
of the singularities of $f$.

\medskip

This spectrum and its relevance in physics, especially in fluid
mechanics, goes back to the 1980's.  At this time, physicists have
been able to measure one coordinate of the velocity of a turbulent
fluid, and they observed that their signals exhibited very different
local behaviors at different times.  This variability  was proposed by
Frisch and Parisi as a possible explanation for the concavity of the
scaling function associated with the velocity   (see \cite{FP} and
several references on the subject). These works are very intimately
related to our questions, since  only the trace of the fluid's
velocity is measured in practice. Hence, to infer some results on the
regularity properties of the  three-dimensional velocity, it is key to
investigate the possible local behavior of traces of Sobolev or Besov
functions. 

\medskip

Precise results on the pointwise regularity of functions belonging to
classical spaces such as Besov $\Besov{s}{p}{\infty}(\R^D) $ spaces
have recently been obtained~\cite{Aubry:zg,
  fraysse:_how_smoot_is_almos_every, jaffard1, jaffard00:_besov}.
These results are of two kinds: universal upper bounds for the
spectrum of singularities (valid for any element of the space) and
almost-sure spectrum (valid for a ``large'' subset of the space, in the
sense of prevalence or Baire categories).  We detail these results, as
well as ours, now.

\medskip

In all that follows, $0<d<D$ are two fixed integers. Let $d' \deq D -
d$ and $(x, x') \in \R^d \times \R^{d'} = \R^D$.  For $a \in \R^{d'}$
we shall denote by $\calH_a \deq \{(x, a)\}$ the $d$-dimensional
affine subspace of $\R^D$.

\medskip

Let $f$ be a continuous function on $\R^D$.  Its trace on $\calH_a$ is
\begin{align*}
  f_a \deq f_{| \calH_a}: & \ \  \R^d \longrightarrow \R \\
  & \ \  x \ \longmapsto f(x, a)
\end{align*} 
If $f$ is not continuous, its trace can be defined by
Fourier regularization: we shall again write $f_a$ for
$\lim\limits_{N\to\infty} \paren{ \Fourier^{-1}
  \paren{ \mathbf{1}_{\abs{\xi}\leq N} \Fourier f} }_a$, whenever that
limit exists.

Standard trace theorems inevitably involve a loss of regularity, for
instance it is well known that when $s > 1/2$, the trace of $f \in
H^s(\R^2)$ on any one-dimensional subspace belongs to $H^{s-1/2}(\R)$.
Similar results hold for Besov spaces (see \S~\ref{sec:besov-spaces}):
it can easily be shown that the trace operator $f \mapsto f_a$ maps
$\Besov{s}{p}{\infty}(\R^D)$ to $\Besov{s-d'/p}{p}{\infty}(\R^d)$.
However, one may expect better regularity properties for most of the
traces $f_a$. Indeed, according to a result of
Jaffard~\cite{jaffard9}:
\begin{theorem}
  \label{theo:tracebesov}
  Let $0 < p, s < \infty$.  If $f \in \Besov{s}{p}{\infty}(\R^D)$,
  then for Lebesgue-almost all $a \in \R^{d'}$ $f_a \in
  \bigcap\limits_{s' < s} \Besov{s'}{p}{\infty}(\R^d)$.
\end{theorem}
In particular, when $s > d/p$, the inclusion
$\Besov{s-\ep}{p}{\infty}(\R^d) \hookrightarrow C^{s-\ep-d/p}(\R^d)$
 for any $\ep>0$ small enough implies that Lebesgue-almost all traces
$f_a$ exist and are uniform H\"older functions.
For the sake of integrity, we present a short and independent proof of
Theorem \ref{theo:tracebesov} in Appendix~\ref{sec:proof-coroll-refc}.

%

\newcommand{\hlevel}[2]{E_{#1}\paren{#2}} 
%

\medskip

By another result of Jaffard~\cite{jaffard1}, belonging to a Besov
space yields an upper bound on the spectrum of singularities:
\begin{theorem}
  \label{theo:upperboundebesov}
  Let $0 < p <  \infty$ and $d/p < s < \infty$.  For any $g \in
  \Besov{s}{p}{\infty}(\R^d)$,  for all $h \geq s-d/p$,
  \begin{equation*}
    d_g(h) \leq \min(d, d + (h-s) p),
  \end{equation*}
  and $E_f(h) = \emptyset $ if $h<s-d/p$.
\end{theorem}
 
\begin{remark}
  The results so far have been stated for Besov spaces
  $\Besov{s}{p}{q}$ with $q = \infty$ but it is clear from classical Besov 
  embeddings (see equation \eqref{eq:Besovembed} below) that they hold identically for any
  $q > 0$.
\end{remark}

 Not only is Theorem~\ref{theo:upperboundebesov} optimal, the upper
bound is actually an \emph{almost sure} equality in   $\Besov{s}{p}{q}(\R^D)$
(Theorem~\ref{theo:genericbesov}) in the sense of prevalence, as
explained below.

%

\medskip

Prevalence theory is used to supersede the notion of Lebesgue measure
in any real or complex topological vector space $E$.  This notion was
proposed by Christensen~\cite{christensen72:_haar_abelian_polis} and
independently by Hunt \textit{et al.}~\cite{hunt92:_preval}.  The
space $E$ is endowed with its Borel $\sigma$-algebra $\Borel(E)$ and
all Borel measures $\mu$ on $(E, \Borel(E))$ will be automatically
\emph{completed}, that is we put $\mu(A) \deq \mu(B)$ if $B \in
\Borel(E)$ and the symmetric difference $A \Delta B$ is included in
some $D \in \Borel(E)$ with $\mu(D) = 0$.  A set is said to be
\emph{universally measurable} if it is measurable for any (completed)
Borel measure.

\begin{definition}
  \label{defprevalence}
  A universally measurable set $A \subset E$ is called \emph{shy} if
  there exists a Borel measure $\mu$ that is positive on some compact
  subset $K$ of $E$ and such that
  $$ \mbox{ for every $x\in E$, } \ \ \mu(A + x) = 0 .$$
  More generally, a set that is included in a shy universally
  measurable set is also called shy.

  Finally, the complement in $E$ of a shy subset is called
  \emph{prevalent}.
\end{definition}

The measure $\mu$ used to show that some subset is shy or prevalent is
called a \emph{probe}.  It can be for instance the Lebesgue measure
carried by some finite-dimensional subspace of $E$: this is the
technique that will be used in \S~\ref{sec:probe-space}.

\medskip

When a set $B$ is prevalent, it is dense in $E$, $B + x$ is also
prevalent for any $x \in E$ and if $(B_n)_{n \in \N}$ is a sequence of
prevalent sets then so is $\bigcap_{n \in \N} B_n$.  Finally, when $E$
has finite dimension, $B$ is prevalent in $E$ if and only if it has
full Lebesgue measure.  This justifies that a prevalent set $B$ is a
``large'' set in $E$ and extends reasonably the notion of full
Lebesgue measure to infinite dimensional spaces.

From now on, without any possible confusion, the term ``almost all''
will be indiscriminately used to describe elements in a prevalent
subset of an infinite-dimensional space, or in a subset having full
Lebesgue measure in a finite-dimensional space.

In this setting, the following was proved by Fraysse and
Jaffard~\cite{fraysse:_how_smoot_is_almos_every}:
\begin{theorem}
  \label{theo:genericbesov}
  Let $0 < p < \infty$, $0 < q \leq \infty$ and $0< s-D / p < \infty$.
  For almost all $g \in \Besov{s}{p}{q}(\R^D)$,
  \begin{equation*}
    d_g(h) =
    \begin{cases}
      D + (h-s) p &\text{ if } h \in [s-D/p, s] \\
      -\infty &\text{ else}
    \end{cases}
  \end{equation*}
  and for $x$ in a set of full Lebesgue measure in $\R^D$, $h_g(x) = s$.
\end{theorem}

\begin{remark}
  Another notion of genericity is given by Baire's theory: a property
  is said to be quasi-sure in a complete metric space $E$ if this
  property is realized on a residual (comeagre) set in $E$.  We choose
  to work within the prevalence framework, but Baire's genericity is
  also worthy of interest and will be studied in a subsequent paper.
\end{remark}

In this paper we prove the following result on the  singularity
spectrum of traces of almost all Besov functions. 
\begin{theorem}
  \label{theo:mainth2}
  Let $0 < p < \infty$, $0 < q \leq \infty$ and $0 < s- d/p < +\infty
  $. For almost all $f$ in $\Besov{s}{p}{q}(\R^D)$, for
  Lebesgue-almost all $a \in \R^{d'}$, the following holds:
  \begin{enumerate}

  \item the spectrum of singularities of $f_a$ is
    \begin{equation}
      \label{eq:uno2}
      d_{f_a}(h) =
      \begin{cases}
        d + (h-s) p &\text{ if } h \in [s-d/p, s] \\
        -\infty &\text{ else.}
      \end{cases}
    \end{equation}
    
  \item for every open set $\Omega \subset \R^d$, the level set
    $\hlevel{f_a}{s}\cap \Omega$ has full Lebesgue measure in
    $\Omega$.

  \end{enumerate}
\end{theorem}

Let us make some remarks on Theorem \ref{theo:mainth2}:

\begin{itemize}[leftmargin=0em,itemindent=1em]

\item In a given Besov space $\Besov{s}{p}{q}(\R^D)$ (Theorem
  \ref{theo:genericbesov}), as well as in $C^\alpha(\R^D)$
  \cite{Jaffard4} or for Borel measures supported by $\zu^D$
  \cite{BucSeu}, the almost-sure regularity is often the ``worst
  possible'', i.e. the upper bound on the spectrum valid for all
  elements of the considered space turns out to be an equality for
  almost all functions or measures. This is not the case in Theorem
  \ref{theo:mainth2}, for which the almost sure spectrum does not
  coincide with the {\em a priori} upper bound, and thus the traces
  are more regular than what could be expected {\em a priori}.

  \medskip

\item Observe that the singularities with H\"older exponent $h$ less
  than $s-d/p$ are ``not seen'' by Lebesgue-almost every traces $f_a$.
  This corresponds to the level sets $E_f(h)$ of Hausdorff dimension
  less than $d'$. B. Mandelbrot referred to this phenomenon as {\em
    negative dimensions}: By this, he means that almost every function
  $f\in\Besov{s}{p}{q}(\R^D)$ possesses singularities with exponent
  $s-D/p\leq h<s-d/p$, but these singularities form a set of too small a
  dimension to intersect a large quantity among the hyperplanes
  $\mathcal{H}_a$ of dimension $d'=D-d$.

\end{itemize}

\begin{figure}
  \begin{center}\includegraphics[width=8.5cm,height=6.5cm]{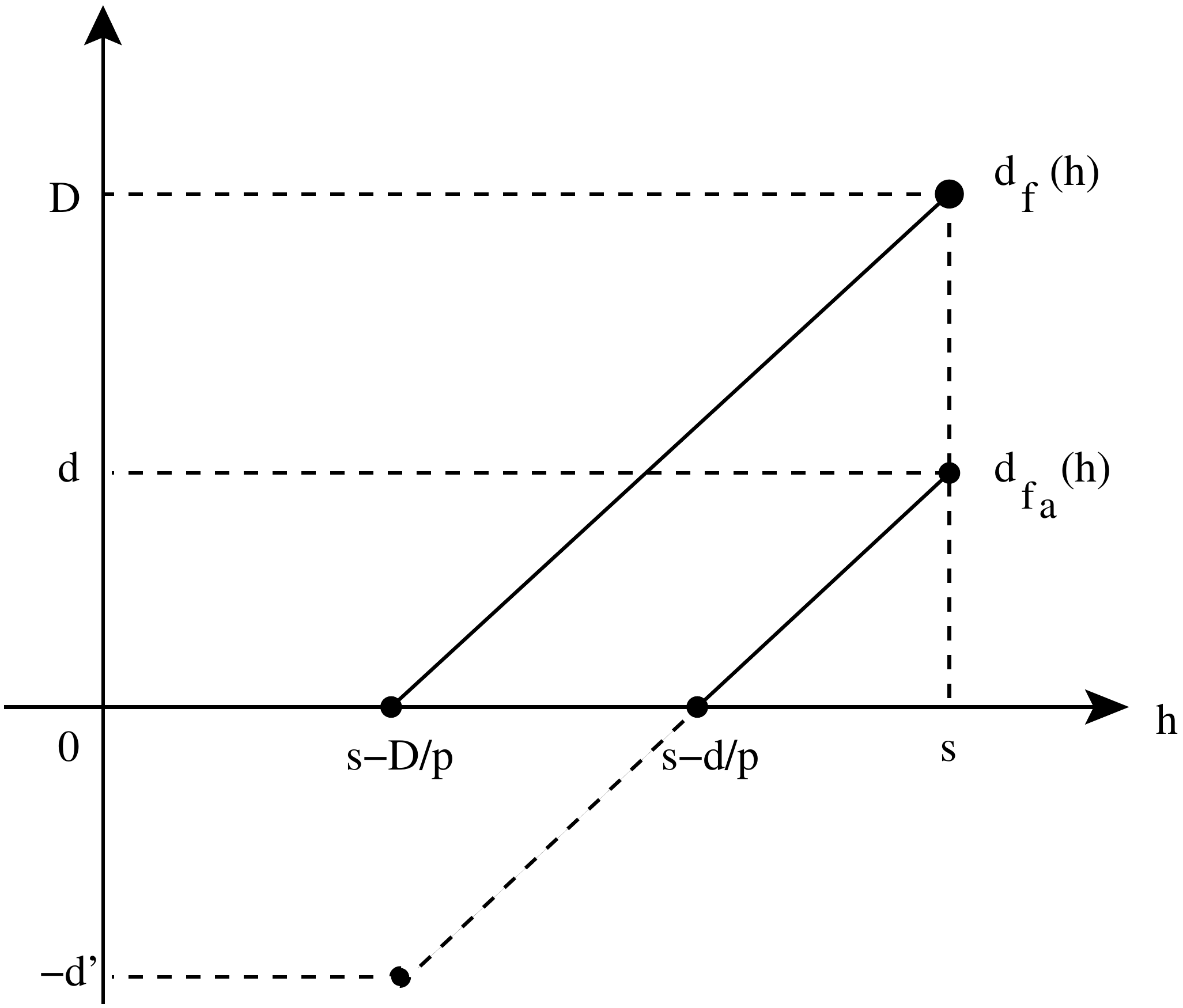}
    \caption{Singularity spectrum of almost all
      $f\in\Besov{s}{p}{q}(\R^D)$ and its trace $f_a$ for Lebesgue
      almost every $a\in \R^{d'}$. }
    \label{fig1}
  \end{center}
\end{figure}

\begin{remark}
  Theorem \ref{theo:mainth2} and the above remark are reminiscent of
  classical results of P. Mattila ~\cite{mattila1} on the Hausdorff
  dimensions of intersection of fractal subset of $\R^D$ with
  Lebesgue-almost all $d$-dimensional hyperplanes, or of sliced
  measures \cite{mattila1,jarvenpaa}.
\end{remark}

In this theorem, all the hyperplanes $\calH_a$ on which the traces are
taken are parallel (to the $d$ first coordinates axes).  Since the
Besov spaces are invariant by unitary transformation of the
coordinates, the result remains valid in any fixed direction.  Thanks
to the stability of prevalence by countable intersection, we thus
obtain: 
\begin{corollary} Let $\Delta$ be a countable subset of the
Grassmannian $\Gr_d(D)$.  Under the same hypotheses on $p, q, s$, for
almost all $f$ in $\Besov{s}{p}{q}(\R^D)$, for any $\calH \in \Delta$,
for Lebesgue-almost all $a \in \mathcal{H}^{\perp}$, the trace of $f$
on $\calH + a$ has the properties stated in
Theorem~\ref{theo:mainth2}.
\end{corollary} 

Unfortunately no Fubini theorem holds for prevalence, so we cannot
directly deduce from this the natural generalization below, which we
leave for subsequent studies.

\begin{conjecture} 
  Consider the Grassmannian $ \Gr_d(D)$ and its Haar measure
  $\mu_{d,D}$.  For almost all $f$ in $\Besov{s}{p}{q}(\R^D)$, for
  $\mu_{d,D}$-almost all $\calH \in \Gr_d(D)$, for Lebesgue-almost all
  $a \in \mathcal{H}^{\perp}$, the trace of $f$ on $\calH + a$ has the
  properties stated in Theorem~\ref{theo:mainth2}.
\end{conjecture} 

The paper is organized as follows. Our method is based on wavelets,
and requires various notions of real and functional analysis. Section
\ref{sec:preliminaries} provides all the definitions and important
results needed to complete the proof of Theorem \ref{theo:mainth}. In
Section \ref{sec:proof-main-result}, we prove the upper bound for the
singularity spectrum for all functions in $\Besov{s}{p}{q}(\R^D)$, and
the lower bound for all functions in a set that we call
$\mathcal{F}$. Then, in Section \ref{sec:proof-keyresult}, we show
that this set $\mathcal{F}$ is prevalent, the main difficulties lying
in the measurability properties of
$\mathcal{F}$. Appendix~\ref{sec:proof-coroll-refc} contains a shorter
proof of Theorem \ref{theo:tracebesov}, and
Appendix~\ref{sec:proof-key-nonsep} deals with the universal
measurability of $\mathcal{F} $ in the case $q=+\infty$ 
(which differs from the case $q < +\infty$ since
$\Besov{s}{p}{\infty}(\R^D)$ is not separable).

\section{Preliminaries}
\label{sec:preliminaries}

\subsection{Dimensions} \ \mk
\label{sec:dimension}

Two notions of dimensions of sets in $\R^d$ will be used below: the
Hausdorff dimension and the upper box dimension. We recall them
quickly.

\mk

Let $E$ be a bounded set in $\R^d$. 
For every $\ep>0$, denote by $N_\ep(E)$ the minimal number of cubes of
size $\ep$ needed to cover the set $E$. The upper box dimension of
$E$, denoted by $\overline{\dim_B}(E)$, is the real number $\in [0,d]
$ defined as 
\begin{equation}
\label{defboite}
 \overline{\dim_B}(E) = \limsup_{\ep \to 0^+} \frac{\log N_\ep(E)}{-\log \ep}.
\end{equation}

For the reader's convenience we also recall the definition of the Hausdorff dimension.

\begin{definition}\label{Hausdorff}
Let $s\ge 0$. The $s$-dimensional Hausdorff measure of a set $E$,
$\mathcal{H}^s(E)$, is defined as
$$
\mathcal{H}^s(E)=\lim_{r\searrow 0}\mathcal{H}_r^s(E),\quad\mbox{with
} \mathcal{H}_r^s(E)=\inf\Big\{\sum_{i}|E_i|^s\Big \},
$$
the infimum being taken over all the countable families of sets $E_i$
such that $|E_i|\le r$ and $E\subset \bigcup_i E_i$. Then, the
Hausdorff dimension of $E$, $\dim_{\calH}\, E$, is defined as 
$$\dim_{\calH}\,
E=\inf\{s\ge 0: \mathcal{H}^s(E)=0\}=\sup\{s\ge 0: \mathcal{H}^s(E)=+\infty\}. $$
\end{definition}

For a bounded set $E\subset \R^d$, we have
$$
0\leq  \dim_{\calH} (E)\leq  \overline{\dim_B}(E) \leq d .$$


\subsection{Wavelets}  \ \mk

\label{sec:wavelets}

\newcommand{\prodprimePsi}{\prod_{i=1}^{d'} \Psi^{l'_i}_{j, k'_i}(a_i)}
\newcommand{\tPsi}{\widetilde{\Psi}}

We recall very briefly the basics of multiresolution wavelet analysis
(for details see for instance~\cite{daubechies92}).  For an arbitrary
integer $N \geq 1$ one can construct compactly supported functions
$\Psi^0 \in C^N(\R)$ (called the scaling function) and $\Psi^1 \in C^N(\R)$ (called the mother wavelet), with $\Psi^1 $ having at least  $N+1$ vanishing moments (i.e. $\int_\R
x^n \Psi^1(x) dx = 0$ for $n \in \{0, \dots, N\}$), and such that the set of functions
\begin{equation*}
 \Psi^1_{j, k}: x \mapsto \Psi^1(2^j x - k)
\end{equation*}
for $j \in \Z, k \in \Z$  form an orthogonal basis of $L^2(\R)$  (note
that we choose   the $L^\infty$ normalization, not~$L^2$).  In this case, the wavelet is said to be $N$-regular.   

\newcommand{\lambdaD}{\lambda^D} 

Let us introduce the notations
\begin{eqnarray*}
   0^d  := (0,0,\cdots , 0),  \ \ \   \ \ \    1^d  := (1,1,\cdots , 1),      \ \ \ \ \ \ \ \  L^d \deq \{0, 1\}^d \backslash 0^d.
\end{eqnarray*}

An orthogonal basis of $L^2(\R^d)$ is then obtained by tensorization. For every $\lambda \deq (j, {\bf k}, {\bf l}) \in \Z\times \Z^d \times \{0,1\}^d$, let us define the tensorized wavelet
\begin{equation*}
  \Psi_{\lambda}(x) \deq   \prod_{i=1}^{d}  \ \Psi^{l_i}_{j,k_i}(x_i) ,
\end{equation*}
 with obvious notations  that $\bfk=(k_1,k_2,\cdots, k_d)$  and $\bfl=(l_1,l_2,\cdots, l_d)$.

Any function $f\in L^2(\R^d)$ can be written (the inequality being true in $L^2(\R^d)$)
\begin{equation}
\label{decompf}
 f =    \ \sum_{\gla=(j,\bfk,\bfl): \, j\in \Z, \, \bfk \in \Z^d,  \, \bfl \in L^d} \ c_{\gla}  \Psi_{\gla } (x),
\end{equation}
where
\begin{equation}
\label{defcoeff}
c_{\gla}     :=     2^{jd} \int_{R^d} f(x)  \Psi_{\gla}(x) \, dx.
\end{equation}
It is implicit in \eqref{defcoeff} that the wavelet coefficients depend on $f$. Observe that in the wavelet decomposition \eqref{decompf}, no wavelet $\Psi_\gla$ such that  $\bfl =0^d$ (where $\gla=(j,\bfk,\bfl)$) appears.

  \mk

Similar notations (e.g.  $\lambda^D \deq (j, (\bfk,\bfk'), (\bfl,\bfl')) \in
\Z\times \Z^D \times  \{0,1\}^D$) with the straightforward modifications will
produce an orthogonal basis of $L^2(\R^D)$.  The wavelets and the corresponding wavelet coefficients in $L^2(\R^D)$ will be  denoted respectively by $ \Psi_{\glaD }$ and $c_{\gla^D} $.

\mk

In~\ref{sec:probe-space-1} we shall 
need to consider the 1-periodic function
\begin{equation}
  \label{eq:defGG}
  G :  t\in \R \longmapsto    \sum_{k \in \Z } \Psi^1(t  - k)  .\end{equation}
and 
make the technical hypothesis on $\Psi^1$:\smallskip
\begin{equation}\smallskip
  \!\!  \tag{$\calH_N$} \label{eq:hypH}
\left\{ \begin{array}{l}\smallskip
 \!\! \text{(i)  \ \ \,  $\Psi^1$ is $N$-regular,}\\ \smallskip
 \!\! \text{(ii) \ \ \   The set } Z \deq
  G^{-1}(\{0\}) \cap \zu \text{ is finite,}
 \\\smallskip
 \!\! \text{(iii) \ \   For every $t\in Z$, $|G^{'}(t)| >0$ .} 
\end{array}\right.
\end{equation}

This condition is very reasonable for a given wavelet $\Psi^1$.
Numerical simulations (see figure \ref{fig2}) indicate that
$(\calH_N)$ is verified for suitable choices of regular wavelets,
including in particular Daubechies's compactly supported
wavelets~\cite{daubechies92}. 
In Figure 2, the simulations of $\Psi^1$ and $(\Psi^1)'$ (computed
using the associated wavelet filters) are precise enough to guarantee
that $G'$ does not vanish around the zeros of $G$.

\begin{figure}
  \begin{center} 
    \includegraphics[width=5.5cm,height=5.5cm]{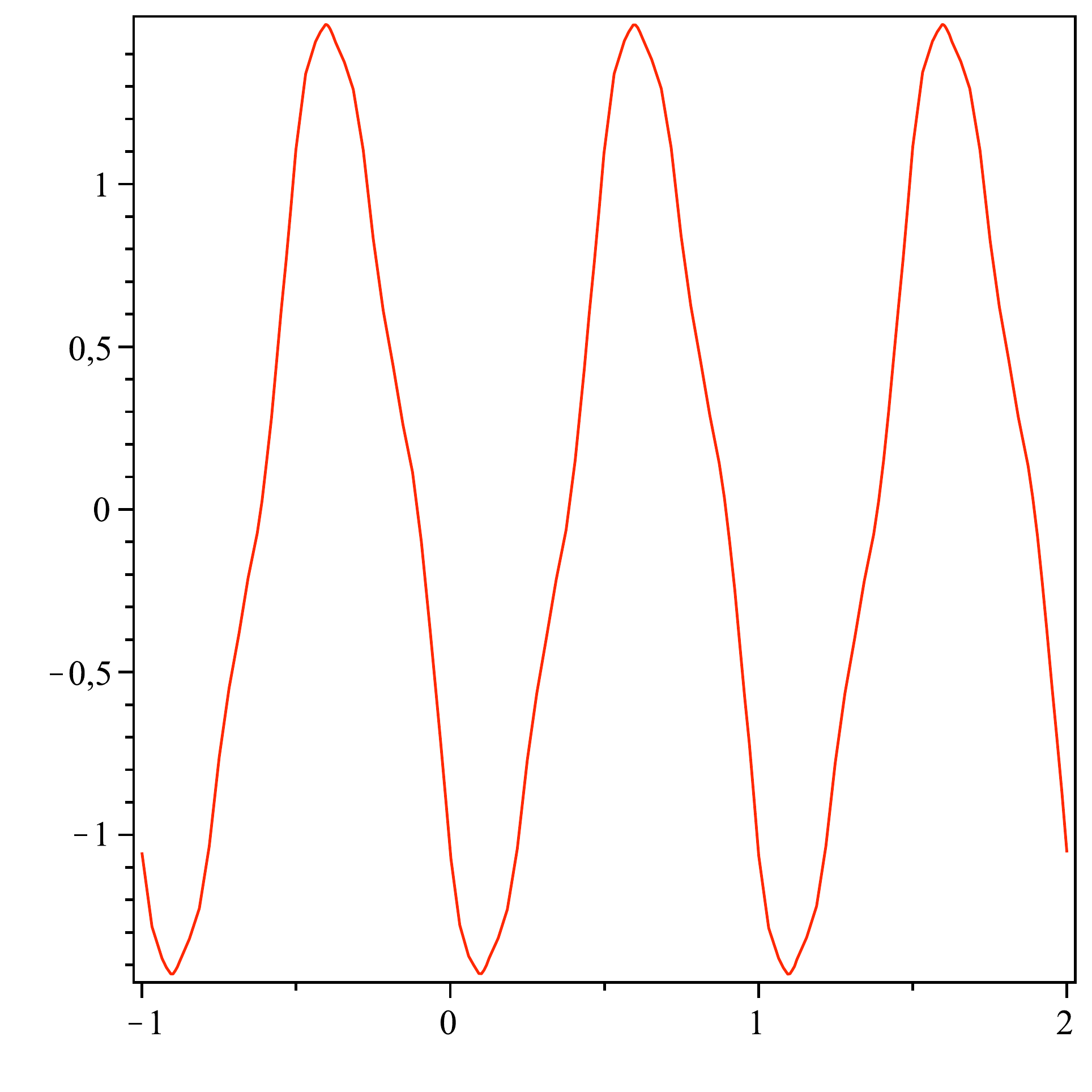}
    \includegraphics[width=5.5cm,height=5.5cm]{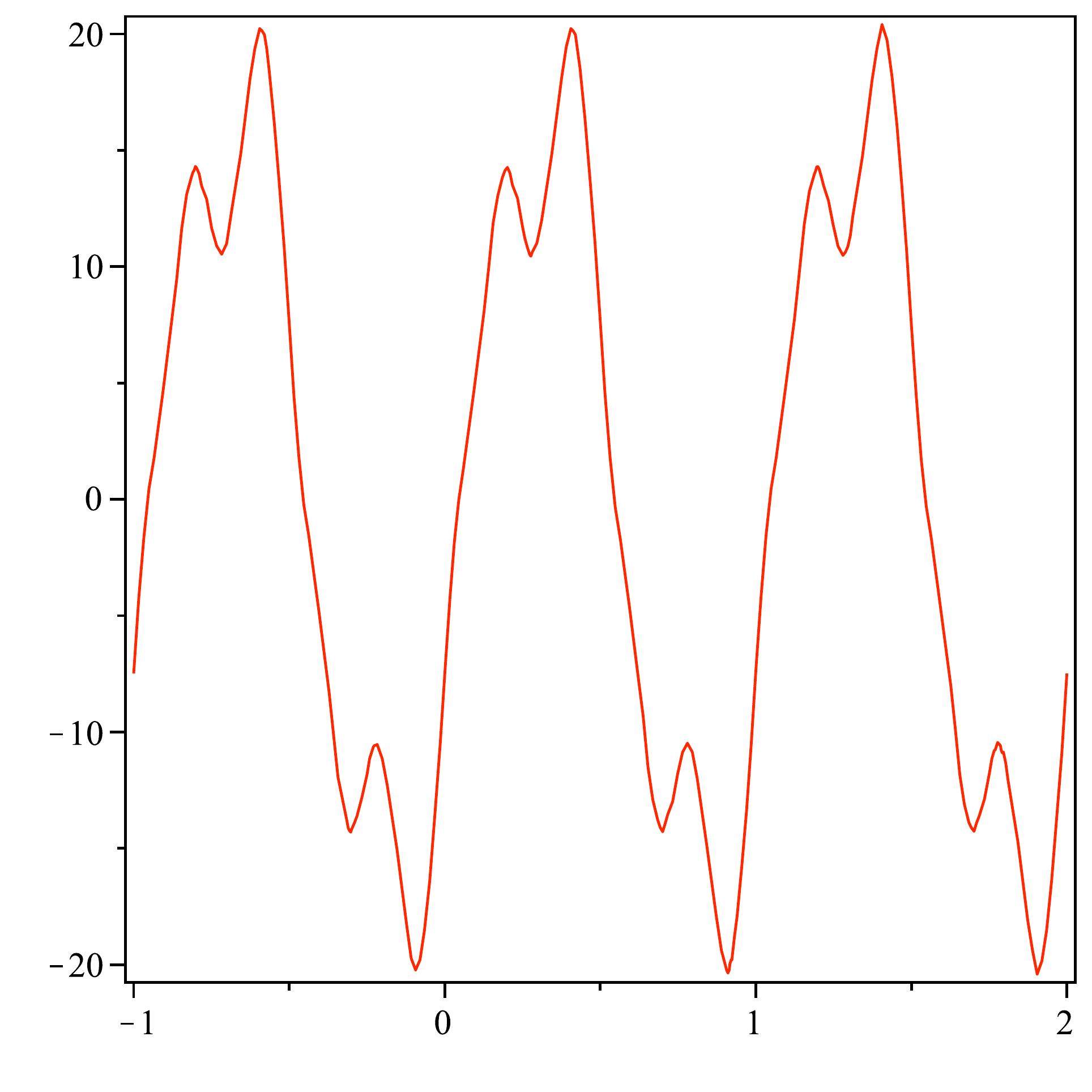}
    \caption{Plot of the periodized Daubechies wavelet with eight
      vanishing moments, and its derivative. The approximation is precise enough to ensure that $G$ and its derivative do not vanish at the same time.}
    \label{fig2}
  \end{center}
\end{figure}

\subsection{Localization of the problem}  \ \mk

\label{sec:period}

We will be first focusing on the local behavior of traces on $(0, 1)^d
\times \{a\}$, $a \in (0, 1)^{d'}$.  
As Proposition~\ref{caracpointwise} shows, if $f$ is written as
\eqref{decompf}, only the coefficients $c_\glaD$ such that $j \geq 0$
and $(\bfk 2^{-j},\bfk' 2^{-j}) \in \zu^D$ can play a role in the value of
the pointwise exponent $h_{f_a}(x)$.  For our purpose, we can identify
functions that have the same wavelet coefficients $c_\glaD$ when
$(\bfk 2^{-j},\bfk' 2^{-j}) \in \zu^D$.  Hence we will consider
functions $f$ of the form
\begin{equation}
  \label{decompf2}
  f = \sum_{\glaD \in \Lambda^D\times L^D} \ c_{\glaD}  \Psi_{\glaD } (x),
\end{equation}
where 
\begin{eqnarray*}
  \mbox{for $ j\geq 1$, }  \ \ \Z_j  & = &  \{ 0,1,\cdots, 2^j-1\} \ \
  \mbox{ and } \ \ \Lambda^D_j = \{j\}\times \Z_j^D\\ 
  \Lambda^D & = &  \bigcup_{j\geq 1} \Lambda^D_j .
\end{eqnarray*}
 
If we prove Theorem \ref{theo:mainth2} on $\zu^D$ instead of $\R^D$,
then by dilation it will be true on any cube $[-N, N]^D$.  Prevalence
results being stable by countable intersection on $N \in \N$, Theorem
\ref{theo:mainth2} will thus be obtained.

We shall present our results in this framework, and we will
effectively prove the following:
\begin{theorem}
  \label{theo:mainth}
  Let $0 < p < \infty$, $0 < q \leq \infty$.  Assuming that $s > d/p$,
  for almost all $f$ in $\Besov{s}{p}{q}(\zu^D)$, for almost all $a \in
  \zu^{d'}$, the following holds:
  \begin{enumerate}

  \item the spectrum of singularities of $f$ is
    \begin{equation}
      \label{eq:uno}
      d_{f_a}(h) =
      \begin{cases}
        d + (h-s) p &\text{ if } h \in [s-d/p, s] \\
        -\infty &\text{ else.}
      \end{cases}
    \end{equation}
    
  \item the level set $\hlevel{f_a}{s}$ has full Lebesgue measure
    in $\zu^d$.

  \end{enumerate}
\end{theorem}

\subsection{Characterization of local and global regularity properties
}  \ \mk

\label{sec:besov-spaces}

Let $0 < s < \infty$, $0 < p, q \leq \infty$.  Assume that the wavelet
$\Psi$ is at least $[s+1]$-regular.  The $\Besov{s}{p}{q}(\zu^D)$
Besov norm (quasi-norm when $p < 1$ or $q < 1$) of a function $f$ on
$\zu^D$ having wavelet coefficients $c_{\lambdaD}$ is defined as
\begin{equation}
  \label{eq:caracbesov}
  \norm{\Besov{s}{p}{q}}{f} = \biggparen{ \sum_{j\geq 1}
    \biggparen{2^{(sp-D) j} \sum_{(\bfk, \bfk') \in \Z_j^D} \abs{c_{\lambdaD}}^p
    }^{\frac{q}{p}}}^{\frac1q}
\end{equation}
with the obvious modifications when $p = \infty$ or $q = \infty$.  The
Besov space $\Besov{s}{p}{q} (\zu^D) $ is naturally the set of
functions with finite (quasi-)norm.  It is a complete metrizable
space, normed when $p$ and $q \geq 1$, separable when both are finite.

The following standard embeddings are easy to deduce
from~(\ref{eq:caracbesov}): 
  For any $0 < s < \infty$, $0 < p \leq \infty$, $0 < q < q' \leq
  \infty$, $\eps > 0$,
  \begin{equation}
    \label{eq:Besovembed}
    \Besov{s}{p}{q} (\zu^D)\hookrightarrow \Besov{s}{p}{q'}(\zu^D)
    \hookrightarrow 
    \Besov{s-\eps}{p}{q}(\zu^D)
  \end{equation}

\begin{remark}
  In contrast with Theorems~\ref{theo:tracebesov} and
  \ref{theo:upperboundebesov}, the prevalence result for a given $q <
  \infty$ cannot simply be deduced from the result for $q = \infty$
  (nor the other way round).  Indeed it can be shown that
  in~(\ref{eq:Besovembed}) each included space is shy in the next one.
\end{remark}

Let us finally recall the fundamental result linking pointwise
regularity and the size of wavelet coefficients, which justifies our
approach.
\begin{proposition} 
  \label{caracpointwise}
  Suppose that $\gamma > 0$ and the wavelet $\Psi$ is at least
  $[\gamma+1]$-regular.  Let $f : \zu^d \to \R$ be a locally bounded
  function with wavelet coefficients $\{c_\lambda\}$, and let $x \in
  \zu^d$.

  If $f \in C^\gamma(x)$, then there exists a constant $M < \infty$
  such that for all $\lambda=(j,\bfk,\bfl) \in \Lambda^d \times L^d$,
  \begin{equation}
    \label{eq:caracholder}
    \abs{c_\lambda} \leq M \paren{2^{-j} + \abs{x - k 2^{-j}} }^\gamma
    = M 2^{-j\gamma} (1+ \abs{2^jx - k } )^\gamma
  \end{equation}
  Reciprocally, if \eqref{eq:caracholder} holds true and if $f\in
  \bigcup_{\ep>0} \, C^\ep(\zu^d)$, then $f \in C^{\gamma -\eta}(x)$,
  for every $\eta>0$.
\end{proposition}

Finally, the notion of cone of influence will be needed later.

\begin{definition}
\label{defcone}
Let $L>0$. The cone of influence of width $L$ above $x \in \R^d$ is
the set of cubes $(j,\textbf{k},\bfl) \in \Lambda^d$  such that  
\begin{displaymath}
\vert x - {\textbf{k}} \,{2^{-j}}\vert\leq  {L}{2^{-j}}.
\end{displaymath}
\end{definition}

\subsection{Traces}  \ \mk

\label{sec:traces}

Recall that for $a \in \T^{d'}$ and $f$ continuous on $\T^D$, the
function $f_a$ is simply defined as $f_a(x) \deq f(x, a)$.  Moreover,
recall that $\gla=(j,\bfk,\bfl)$ with $j\in \N^*$, $\bfk \in \Z^d_j $
and $\bfl \in \{0,1\}^d$ and that $\glaD=(j,(\bfk,\bfk'),
(\bfl,\bfl'))$ with $j\in \N^*$, $\bfk \in \Z^d_j $, $\bfk' \in
\Z^{d'}_j $, $\bfl \in \{0,1\}^d$ and $\bfl' \in \{0,1\}^{d'}$. Using
the expansion \eqref{decompf2} of $f$ in the tensorized wavelet basis
$\accol{\Psi_{\lambda^D}}$, we have
\begin{align}
\nonumber
  f_a(x) &=  \sum_{\lambdaD \in \Lambda^D \times L^D} c_{\lambdaD}
  \prod_{i=1}^{d} \Psi^{l_i}_{j,k_i}(x_i) \prodprimePsi  \\
  \label{eq:trtrtr} &=  G_a(x) + F_a(x)   \end{align}
where
\begin{eqnarray}
  \label{eq:tracecoef}
  G_a(x) & := & \sum_{\lambda \in \Lambda^d \times 0^d }
  d_\lambda(a)   \Psi_\lambda(x) 
  \\
  \label{eq:tracecoef2} F_a(x)  & :=  &   \sum_{\lambda \in \Lambda^d
    \times L^d} d_\lambda(a) 
  \Psi_\lambda(x)
\end{eqnarray}
and for $\gla=(j,\bfk,\bfl) \in \Lambda^d\times \{0,1\}^d$, 
\begin{eqnarray}
  \label{eq:tracecoef4}  \mbox{if $\bfl=0^d$, } & d_\lambda(a)    \deq
  & \sum_{ \substack{ \glaD =(j,(\bfk,\bfk'),(0^d,\bfl')): \\   \bfk'
      \in \Z_j^{d'},  \,  \bfl' \in L^{d'}}} 
  c_{\lambdaD} \prodprimePsi.\\
 \label{eq:tracecoef3} \mbox{if $\bfl  \in L^d$,  } & d_\lambda(a)
 \deq  & \sum_{ \substack{\glaD =(j,(\bfk,\bfk'),(\bfl,\bfl')): \\
     \,   \,\bfk' \in \Z_j^{d'},  \, \bfl' \in \{0, 1\}^{d'}}} 
  c_{\lambdaD} \prodprimePsi.
\end{eqnarray}
Formula \eqref{eq:tracecoef2} indeed yields a wavelet decomposition of
the  function $F_a$, since the wavelets appearing in
\eqref{eq:tracecoef2}  form a wavelet basis of $L^2(\T^d)$ (if
completed by the  function $\Psi_{0,0^d,0^d}$). This is not the case
for the function $G_a$ with formula \eqref{eq:tracecoef}, since only
the scaling function  $\Psi^0$ is used in this
decomposition. Fortunately, we have the following standard result for
the Besov properties of a function $G_a$ defined through a formula
like \eqref{eq:tracecoef}. 

\begin{proposition}
  \label{proptrace1}
  If $s_0>0$, and $g(x) = \sum_{\lambda \in \Lambda^d \times 0^d}
  d_\lambda \Psi_\lambda(x)$ with $\{d_\gla\}$ satisfying
  \begin{equation}
    \label{maj0}
    \sup _{j \geq 1}  \ \ 2^{j( p_0 s_0-d )}  \left (
      \sum_{\gla=(j,{\bf k}, {\bf l}): \,  {\bf k} \in \Z_j^d,  \,
        {\bf l}=0^d} |d_\lambda|^p \right) < + \infty , 
  \end{equation}
  then $g\in B^{s_0}_{p_0,\infty}(\zu^d)$.
\end{proposition}
Proposition \ref{proptrace1} entails that the same Besov
characterization as \eqref{eq:caracbesov} when one considers only
scaling functions. The proof of Proposition \ref{proptrace1}, that we
do not reproduce here, 
consists of decomposing each scaling function
$\Psi_\gla$, for $\gla = (j,\bfk,0^d)$ on the wavelets of smaller
frequencies, i.e. on $\Psi_{\widetilde\gla}$ with $\widetilde\gla =
(\widetilde j,\widetilde\bfk, \widetilde \bfl )$ such that $\widetilde
j \leq j $ and  $\widetilde \bfl \in L^d$.

\smallskip


As a conclusion, the trace $f_a$ can be written
\begin{equation}
  \label{decompffinal}
  f_a  =   \sum_{\lambda \in \Lambda^d \times \{0,1\}^d} d_\lambda(a)
  \Psi_\lambda(x)
\end{equation}
where for $\gla=(j,\bfk,\bfl) \in \Lambda^d\times \{0,1\}^d$,
$d_\gla(a)$ is given by \eqref{eq:tracecoef4} and
\eqref{eq:tracecoef3}. For such a decomposition, the Besov
characterization \eqref{eq:caracbesov} holds true, the difference with
\eqref{decompf} is that the sum over $\gla\in L^d$ is replaced by
$\gla\in \{0,1\}^d$.

Recalling now Theorem \ref{theo:tracebesov} (proved in
Appendix~\ref{sec:proof-coroll-refc}), $f_a \in \bigcap_{\ep>0}
\Besov{s-\ep}{p}{\infty}(\zu^d)$ for Lebesgue-almost every
$a\in\zu^d$.  Hence, still for almost every $a$, we can consider the
{\em effective} wavelet decomposition of $f_a$ on the wavelet basis
provided by \eqref{decompf}, and we write
\begin{equation}
  \label{decompffinal2}
  f_a  =   c_{0,0^d,0^d } \Psi_{0,0^d,0^d}(x) +  \sum_{\lambda \in
    \Lambda^d \times L^d} c_\lambda(a) 
  \Psi_\lambda(x).
\end{equation}

We will use both forms \eqref{decompffinal} and \eqref{decompffinal2}.

\subsection{Dyadic approximation}  \ \mk

\label{sec:dyadic-approximation}
 
Let $B(x, r)$ denote the closed $l^\infty$ ball of radius $r$ around
$x$ in $\T^d$.  For $\alpha \geq 1$ and $j \in \N$, let
\begin{eqnarray*}
  \calX^{\alpha}_{j}  & \deq &  \bigcup_{k\in\Z_j^d} B(k 2^{-j}, 2^{-j
    \alpha})\\ 
  \mbox{and } \ \ \ \ 
  \calX^{\alpha}  & \deq &   \limsup_{j\rightarrow\infty} \calX^{\alpha}_{j}
\end{eqnarray*}
The set $ \calX^{\alpha} $ is constituted by points in $\T^d$ that are
approached at rate at least $\alpha$ by dyadics.  In other words, $x
\in \calX^\alpha$ if and only if there exists a sequence $(J_{n},
K_{n})_{n\geq 1} \in \Lambda^d$ such that ${J}_n \to +\infty$ and for
all $n \in \N$
\begin{equation}
  \label{eq:eq311}
   \abs{x - K_n 2^{-J_n} } \leq {2^{- \alpha J_n}}. 
\end{equation}  
Observe that $\calX^1 = \T^d$ and if $\alpha \leq \alpha'$ then
$\calX^{\alpha'} \subset \calX^{\alpha}$. Observe also that if $x \in
\calX^\alpha$ is not itself a dyadic, then the sequence $(J_{n},
K_{n})$ can be chosen so that for every $n$ the fraction
$\frac{\bfK_n}{2^{J_n}} $ is irreducible.  We call $(J_{n},
K_{n})_{n\geq 1} $ an {\em irreducible} sequence.

  About the  dimension of $\calX^\alpha$, a well know
result (for instance proved in ~\cite{Fa1}) states:
    
\begin{theorem}
  \label{theo:ubiquite}
  There exists a positive $\sigma$-finite measure $m_\alpha$ carried by
  $\calX^\alpha$ and such that any set $E$ having Hausdorff dimension
  $\dim_{\calH}(E) < \frac{d}{\alpha}$ has measure $m_\alpha(E)=0$.

  In particular, $m_\alpha(\calX^\alpha)>0$ and $\dim_\calH
  \calX^\alpha = d/\alpha$. 
\end{theorem}

\subsection{Prevalence, universal measurability, analytic sets} \ \mk
\label{sec:prevalence}

In the Definition \ref{defprevalence} of the prevalence in a complete
metric space $E$, a set $B\subset E$ needs to be universally
measurable to be shy or prevalent (this includes the Borel sets). One
main difficulty occurring in the proof of Theorem \ref{theo:mainth2}
lies in the universal measurability property of subsets of $E$ for
which we aim to prove a prevalence property. Indeed, these sets will
be defined through complicated formulas, not easily tractable. In
particular, these subsets of $E$ can often be viewed as continuous
images of Borel sets.

When $E$ is a Polish space (this is the case for $B^s_{p,q}(\R^D)$
when $q<\infty$), such sets are called \emph{analytic}, and we have
the following theorem \cite{Choquet:1954uo}.

\begin{theorem}
  \label{theo:unimeasurable}
  Every analytic set in a Polish space is universally measurable.
\end{theorem}

When $E$ is not Polish (in our context, when
$E = B^s_{p,\infty}(\R^D)$), continuous images of Borel sets need not be
universally measurable. Hence, in order to obtain the universal
measurability for our specific sets, the definition of an analytic set
has to be modified and is more complicated (see
\S~\ref{sec:analytic-sets}). Once this second definition is adopted,
the same result as Theorem \ref{theo:unimeasurable} holds,
i.e. analyticity implies universal measurability. The fact that the
sets we will meet indeed satisfy this second definition of analytic
set is proved in \S~\ref{sec:measurability-nonsep}.

\section{Proof of Theorem \ref{theo:mainth}}
\label{sec:proof-main-result}

\subsection{A first property of the wavelet}  \ \mk

\label{sec:hypotheseH}

Recall the definition    \eqref{eq:defGG} of the periodized wavelet
$G$ and  let us introduce its $d'$-dimensional version $G_{d'}$
defined as 
\begin{equation}
\label{defgdprime}
    G_{d'}:  x \in \R^{d'} \longmapsto    G(x_1)\cdot G(x_2) \cdots G(x_{d'}).
    \end{equation}

\begin{proposition}
  \label{prop:protr1}
  If $\Psi$ satisfies  \eqref{eq:hypH}, then the set 
  \begin{displaymath}
    \Aun \deq \set{a\in \T^{d'}}{\exists \, j_a
      \in \N \textrm{ such that } \,\forall j \geq j_a, \, 
      \abs{G_{d'}(2^j a)} >   j^{- 2d' } }
  \end{displaymath}
  has full Lebesgue measure.
\end{proposition}

Remark that Proposition \ref{prop:protr1}  holds  in fact  for any function $G \in C^N(\R)$ satisfying assumptions (ii) and (iii) of \eqref{eq:hypH}.

\begin{proof}
Obviously, if we are able to prove that   the set 
 \begin{displaymath}
    \Aun' \deq \set{a_1\in \T }{\exists \, j_a
      \in \N \textrm{ such that } \,\forall j \geq j_a, \, 
      \abs{G (2^j a_1)} >   j^{-2  } }
  \end{displaymath}
has full Lebesgue measure in $\zu$, then Proposition \ref{prop:protr1} will be proved  since we have the inclusion $   (\Aun')^{d'} \subset    \Aun$.

\mk
   
By \eqref{eq:hypH}, there is a finite number, say $y_1$, $y_2$, $\cdots$, $y_p$, of  zeros of $G$ on the interval $\zu$, and $G'$ does not vanish at these real numbers. Let $M=\min(|G'(y_1)|,|G'(y_2)|, \cdots, |G'(y_p)|)$. For each $y_i$ there is a small interval $[y_i-r_i,y_i+r_i]$ around $y_i$ on which $|G(a)|\geq M/2 |a-y_i|  $.

Let $r=\min(r_i: i=1,..., p)$.


Let us denote by $m$ the minimum of $G$ on the compact set $\zu\setminus \bigcup_{i=1}^p (y_i - r,y_i+r)$. We now choose an integer $n_1 $ such that $1/n_1 \leq \min(m,r)$.

\mk

The above construction guarantees that for every  integer $n\geq n_1$, for every $a\notin \bigcup_{i=1}^p [ y_i - \frac{2}{M\cdot n }    ,y_i+  \frac{2}{M\cdot n } ]$, we have $|G(a)|\geq 1/n$. In other words, $|G(a)| < 1/n$ except on a set of at most Lebesgue measure  $\sum_{i=1}^p 2 \frac{2}{M\cdot n } = C/n$, for some  constant $C>0$.  This immediately implies that for every $j$ large enough, the set
  \begin{displaymath}
    \widetilde{\calA}(j) \deq \{a \in \T : 
      \abs{G (  a )}  \leq  j^{ -2}  \}
  \end{displaymath}
has a Lebesgue measure less than $C j^{-2}$.

Remarking the 1-periodicity of $G$, we deduce that the Lebesgue measure of the set 
   \begin{displaymath}
 {\calA'}(j) \deq \{a \in \T : 
      \abs{G (  2^j a )}  \leq  j^{-2 }  \}
  \end{displaymath}
  is also equal $C j^{-2}$ (the same as that of $   \widetilde{\calA}(j)  $). 
    Obviously,
    $$\sum_{j\geq 1}  \ \mathcal{L}({\calA'}(j) ) <+\infty.$$
  Thus, applying the Borel-Cantelli lemma to the sets $ {\calA'}(j)$, we deduce that the limsup set  
  \begin{displaymath}
  \bigcap_{J\geq 1} \ \ \bigcup_{j\geq J}  \ \  \calA'(j) 
      \end{displaymath}
has zero Lebesgue measure. This set is the complement of the set $\calA'_1$, which by deduction is of full Lebesgue measure in $\zu$.
\end{proof}

\subsection{Prevalence property of an ancillary set}  \ \mk

\label{sec:key-result}

The key result to obtain the prevalence of the  singularity spectrum of Theorem \ref{theo:mainth} is the following theorem.

\begin{theorem}
  \label{theo:keyresult}
  Suppose that $0 < s- D/p < \infty$ and $0 < q \leq \infty$.  Let $\alpha \geq
  1$ and let us defined the exponent
  \begin{equation}
\label{defhalpha}
H(\alpha) \deq s - \frac{d}{p} + \frac{d}{\alpha p}.
\end{equation}
The
  set
  \begin{multline*}
    \calF_\alpha \deq \left\{f \in \Besov{s}{p}{q}(\T^D)
      \suchthat \exists  \, \calA(f) \text{ of full Lebesgue measure such that}
    \right. \\    a \in \calA(f) \Longrightarrow \forall x \in
      \calX^\alpha,  \ h_{f_a}(x) \leq H(\alpha)     \Big\}
        \end{multline*}
  is prevalent in $\Besov{s}{p}{q}(\T^D)$.
\end{theorem}

The proof of Theorem \ref{theo:keyresult} is postponed to
\S~\ref{sec:proof-keyresult}.  We admit it for the
moment, and we explain how we conclude once Theorem
\ref{theo:keyresult} is proved.  Our main result,
Theorem~\ref{theo:mainth}, is a direct consequence of
Propositions~\ref{prop:upperbound}--\ref{prop:lowerboundtrace} below.

\mk

From
now on, let $(\alpha_n)_{n\in\N}$ be a dense sequence in $[1, \infty)$
such that $\alpha_0 = 1$.  Using the fact that a countable
intersection of prevalent (resp. full Lebesgue measure) sets is
prevalent (resp. of  full Lebesgue measure), it follows immediately that:

\begin{corollary}
  \label{coro:keyresultdense}
  The set
  \begin{multline*}
  \!\!\!\!\!   \calF \deq \left\{f \in \Besov{s}{p}{q}(\T^D) \suchthat
      \exists  \, \calA(f) \text{ of full Lebesgue measure such that}
    \right. \\   a \in \calA(f) \Rightarrow \forall n\in\N, \, 
      \forall x \in \calX^{\alpha_n}, \, h_{f_a}(x) \leq H(\alpha_n)
    \Big\}
  \end{multline*}
  is prevalent in $\Besov{s}{p}{q}(\T^D)$.
\end{corollary}

\subsection{Prevalent upper bound}  \ \mk

\label{sec:prev-upper-bound}

We first find an upper bound for the singularity spectrum of Lebesgue
almost traces of $f$, for {\bf every $f \in \Besov{s}{p}{q}(\T^D)$}. 
\begin{proposition}
  \label{prop:upperbound}
  For every $f \in \Besov{s}{p}{q}(\T^D)$, for almost all $a \in
  \T^{d'}$, we have 
   \begin{equation}
   \label{equpper}
  \mbox{ for every $h\geq s-d/p$, } \ \ \   d_{f_a}(h) \leq \min(d, d
  + (h-s) p ). 
  \end{equation}
  \end{proposition}
\begin{proof}
  Let $f \in \Besov{s}{p}{q} (\R^D)$. By
  Theorem~\ref{theo:tracebesov}, there is a set $\calA(f)$ of full
  Lebesgue measure in $ \zu^{d'}$ such that for every $a\in \calA(f)$,
  the trace $f_a$ belongs to $ \bigcap\limits_{s-\ep< s} \Besov{s
    -\ep}{p}{\infty}(\R^d)$. Then, by
  Theorem~\ref{theo:upperboundebesov}, for every $h\geq s-d/p$, for
  every $\ep>0$,
 \begin{equation*}
    d_{f_a}(h) \leq \min(d, d + (h-s) p+\ep p).
  \end{equation*}
  Moreover, for every $\ep>0$, since $f_a \in
  \Besov{s-\ep}{p}{\infty}(\R^d)$ for every $a\in \calA(f)$, there is
  no point $x\in \zu^d$ such that $h_{f_a}(x) <s -d/p-\ep$.

  Letting $\ep>0$ yields exactly  the   upper bound \eqref{equpper}.
\end{proof}
 
One can obtain more precise informations for {\bf almost all $f \in
  \Besov{s}{p}{q}(\T^D)$}, i.e. on a prevalent set in $f \in
\Besov{s}{p}{q}(\T^D)$.

\begin{proposition}
  \label{prop:largestHoldertrace}
  For almost all $f \in \Besov{s}{p}{q}(\T^D)$, for Lebesgue-almost
  all $a \in \T^{d'}$, for all $x \in \T^d$, $h_{f_a}(x) \leq s$.
\end{proposition}
\begin{proof}
   We apply  Corollary~\ref{coro:keyresultdense} with $\alpha_n =
   \alpha_0=1$: if $f$ belongs 
  to the prevalent set $\calF$, then for any $a \in \calA(f)$,
  for any $x \in \calX^{\alpha_0} = \calX^1 = \T^d$, $h_{f_a}(x) \leq
  H(\alpha_0) = s$.
\end{proof}

\subsection{Prevalent lower bound}  \ \mk

\label{sec:prev-lower-bound}

\begin{proposition}
  \label{prop:lowerboundtrace}
  For almost all $f \in \Besov{s}{p}{q}(\T^D)$, for almost all $a \in
  \T^{d'}$, for any $h \in [s - d/p, s]$, $d_{f_a}(h) \geq d + (h-s)p$
  and furthermore, $\hlevel{f_a}{s}$ has full Lebesgue measure.
\end{proposition}
\begin{proof}
  Consider a function $f$ in the prevalent set $\mathcal{F}$. Let $h \in (s - d/p, s]$. This exponent can be written 
  \begin{equation}
\label{eqhalpha}
 h=H(\alpha)=s-\frac{d}{p} +\frac{d}{\alpha p}
\end{equation}
  for some given $\alpha \geq 1$. 
  
  Consider a subsequence     $(\alpha_{\phi(n)})_{n\in\N}$ of $(\alpha_n)_{n\in\N}$ which  is nondecreasing and converges to   $\alpha$ (for $\alpha = 1$ this would just be $\phi = 0$).

  \mk
  
  Let us first assume that $\alpha>1$, i.e. $H(\alpha) \in (s-d/p,s)$. 
  Remark that $ \calX^{\alpha} \subset   \bigcap _{n\geq 1} \calX^{\alpha_{\phi(n)}}$.  Since $f \in
  \calF$, it follows that for all $a \in \calA(f)$ and $x \in
  \calX^{\alpha}$, $h_{f_a}(x) \leq H(\alpha)$. Hence $\calX^\alpha \subset \{x: h_{f_a(x)} \leq H(\alpha) \}$.
  
  \mk
  Recall that Theorem \ref{theo:ubiquite} provides us with a measure $m_\alpha$ which is supported by $\calX^\alpha$, and which gives measure 0 to every set of dimension strictly less than $d/\alpha$.
  
Let us introduce the set
  $\calY^\alpha \deq \set{x}{h_{f_a(x)} < H(\alpha)}$.  Clearly,
  \begin{equation*}
    \calY^\alpha = \bigcup_{n\geq 1}   \set{x }{h_{f_a}(x )\leq H(\alpha) - {1}/{n} }.
  \end{equation*}
 
 By \eqref{equpper}, each set $ \set{x}{h_{f_a}(x )\leq H(\alpha) - {1}/{n} }$ has Hausdorff dimension strictly less than $d/\alpha$. 
   The scaling properties  and the $\sigma$-additivity of the measure $m_\alpha$ 
yield that  $m_\alpha\paren{\calY^\alpha} = 0$.  

Remembering that $m_\alpha(\calX^\alpha)>0$,   we have  $m_\alpha(\calX^\alpha \backslash \calY^\alpha) > 0$. This means equivalently  that  $m_\alpha(\set{x
      \in  \calX^\alpha}{h_{f_a}(x ) = H(\alpha)  } ) > 0$. This implies that the set $\set{x
      \in  \calX^\alpha}{h_{f_a}(x ) = H(\alpha)  }$ has Hausdorff dimension greater than $d/\alpha$, and thus
      $$d_{f_a}(h) = d_{f_a}(H(\alpha))=\dim_{\calH}\{x:h_{f_a}(x) =H(\alpha)\} \geq d/ \alpha =  p(h-s) +d,$$
      the last equality following from \eqref{eqhalpha}.
      
      \mk
      
When $\alpha = 1$,
  the same reasoning using the $d$ -dimensional Lebesgue measure $\Lebesgue_d$ instead of $m_\alpha$ yields
  $\hlevel{f_a}{s} \supset \T^d \backslash \calY^1$ with
  $\Lebesgue_d(\calY^1) = 0$. Hence $\Lebesgue_d (\hlevel{f_a}{s}) =1$.

\mk

  Finally, it remains us to treat the case of the smallest exponent  $h = s - d/p$. Remembering the definition of $\mathcal{F}$, observe that at  any element  $x$ of
  $\calX^\infty \deq \bigcap_{\alpha\geq 1} \calX^\alpha = \bigcap_{ n \geq 1} \calX^{\alpha_n} $,  one necessarily has $h_f(x) \leq s-d/p$.  Since the converse inequality   holds true for any $x$,  we have proved that $\calX^\infty \subset E_{f_a}(s-d/p)$. We conclude by noting that $\calX^\infty$ is
  certainly not empty (and  uncountable), since it is  a dense $G_\delta$ set in $\R^d$.
\end{proof}

Theorem~\ref{theo:mainth} is now proved, provided that we can  establish
Theorem~\ref{theo:keyresult}.

\section{Proof of Theorem~\ref{theo:keyresult}:  prevalence of
  $\calF_\alpha$} 
\label{sec:proof-keyresult}

\newcommand{\Fagn}{\calO_{\gamma, N}}
\newcommand{\Fagnn}{\calO_{\gamma_n, N}}
\newcommand{\Aagn}{\calA_{\gamma, N}(f)}

We simplify the problem by including the complement of $\calF_\alpha$
in a countable union of simpler ancillary sets.  Let $N$ be an
integer, $\alpha > 1$, $\gamma > H(\alpha)$ and
\begin{equation*}
  \Fagn \deq \set{f\in \Besov{s}{p}{q}(\T^D)}{\Lebesgue_{d'}(\Aagn) > 0}
\end{equation*}
where
\begin{equation*}
  \Aagn = \set{a \in \T^{d'}}{ 
    \begin{array}{c}
      \exists x \in  \calX^\alpha, \ \forall \lambda = (j, \bfk, l) \in
      \Lambda^d \times L^d, \\
      \abs{c_{\lambda}(a)} \leq N 2^{-\gamma j} (1 + \abs{2^j x -
        \bfk})^\gamma
    \end{array} } .
\end{equation*}

Remark that the conditions on the wavelet coefficients that appear in the definition of $ \Aagn$ implies that $f_a$ has exponent greater than $\gamma$ at $x$.

Recall the definition of  $  \calF_\alpha$
 \begin{multline*}
    \calF_\alpha \deq \left\{f \in \Besov{s}{p}{q}(\T^D)
      \suchthat \exists  \, \calA(f) \text{ of full Lebesgue measure such that}
    \right. \\    a \in \calA(f) \Longrightarrow \forall x \in
      \calX^\alpha,  \ h_{f_a}(x) \leq H(\alpha)     \Big\}.
        \end{multline*}
\begin{proposition}
  For any sequence $(\gamma_n)_{n\in\N}$ strictly decreasing to $H(\alpha)$,  we
  have
  \begin{equation*}
    \Besov{s}{p}{q}(\T^D) \backslash \calF_\alpha \subset
    \bigcup_{n, N \in \N} \Fagnn
  \end{equation*}
\end{proposition}

\begin{proof}
We   write that  $\calF_\alpha = \bigcap_{n \geq 1} \calF_{\alpha,
    \gamma_n}$, where for any $\gamma > H(\alpha)$ we put
  \begin{displaymath}
    \calF_{\alpha, \gamma} = \set{f\in
    \Besov{s}{p}{q}(\T^D)}{\begin{array}{c}   
        \exists  \, \calA_\gamma(f)  \text{ of full Lebesgue
          measure such that  } \\
    a \in  \calA_\gamma(f) \Rightarrow   \  \forall x \in  \calX^\alpha, \ 
        f_a \notin C^{\gamma}(x)
      \end{array} }.
  \end{displaymath}
When $f \not \in \calF_{\alpha, \gamma}$,  the set
  $\set{a \in \T^{d'}}{\exists x \in \calX^\alpha, f_a \in
    C^\gamma(x)}$ has positive Lebesgue measure. But
  by~(\ref{eq:caracholder}) of Proposition \ref{caracpointwise} which gives the characterization of $C^\gamma(x)$ in terms of wavelet coefficients, this last set is included in $   \Aagn$, for some $N\geq 1$.  Hence   $\Besov{s}{p}{q}(\T^D) \backslash \calF_{\alpha,
    \gamma} \subset \bigcup_{N \in \N^*} \Fagn$ and
  the conclusion follows.
\end{proof}

To prove Theorem~\ref{theo:keyresult}, it suffices now to show that
each set $\Fagn$ is universally measurable
(Proposition~\ref{prop:mesure}) and shy
(Proposition~\ref{prop:timide}).

From now on we fix $N \in \N^*$, $\alpha > 1$ and $\gamma > H(\alpha)$.

\subsection{Measurability}  \ \mk

\label{sec:measurability}

First we deal here only with the case $q < \infty$, that is
when $\Besov{s}{p}{q} (\R^D)$ is a Polish space.  The case $q = \infty$ is
proved in Appendix~\ref{sec:proof-key-nonsep},
Proposition~\ref{prop:mesure-nonsep}.

\begin{proposition}
  \label{prop:mesure}
  The set $\Fagn$ is universally measurable in $\Besov{s}{p}{q}(\T^D)$. 
\end{proposition}

\newcommand{\Phit}{\widetilde\Phi}
\begin{proof}
  Let
  \begin{align*}
    \Phi_{\lambda}(f, a, x) &\deq N 2^{-\gamma j} (1 + \abs{2^jx
      - k})^\gamma - \abs{c_{\lambda}(a)} \\
    \Phi(f, a, x) &\deq \inf_{\lambda \in \Lambda^d \times L^d}
    \Phi_{\lambda}(f, a, x) \\ \intertext{and} \Phit(f) &\deq
    \Lebesgue_{d'}\paren{ \set{a \in \T^{d'} }{ \exists x \in
        \calX^\alpha, \Phi(f, a, x) \geq 0 } }
  \end{align*}
  so that
  \begin{equation*}
    \Fagn =  \Phit^{-1}  \big( (0, +\infty) \big).
  \end{equation*}
To obtain Proposition \ref{prop:mesure},  we just have to prove that $\Phit$ is universally measurable as a map $:\Besov{s}{p}{q}(\T^D) \longrightarrow \R^+$. For this, let us fix a complete Borel measure $\mu$ on
  $\Besov{s}{p}{q}(\T^D)$.
  
First observe  that each $\Phi_\lambda$ is continuous on
  $\Besov{s}{p}{q}(\T^D) \times \T^{d'} \times \T^d$. This follows from the fact that the dependence of the wavelet coefficients $c_\glaD$ on $f$ is continuous (the Besov topology
  induced on the space of wavelet coefficients
  by~(\ref{eq:caracbesov}) is stronger than the product topology) and the dependence of $c_\gla(a)$ on the variables $a$ and $c_\glaD$ is also continuous (from their definitions \eqref{eq:tracecoef4}, \eqref{eq:tracecoef3} and \eqref{decompffinal2}).  As
  a countable infimum of continuous functions, $\Phi$ is Borel on the
  Polish space $\Besov{s}{p}{q}(\T^D) \times \T^{d'} \times \T^d$.

  Clearly $\calX^\alpha \in \Borel(\T^d)$, so the set
  \begin{equation*}
    \widetilde{\calT} \deq \Phi^{-1}([0, \infty)) \cap
    \Big(\Besov{s}{p}{q}(\T^D)
     \times \T^{d'} \times \calX^\alpha \Big)
  \end{equation*}
  is also Borel and its projection along the third coordinate
  \begin{equation*}
    \Pi(\widetilde{\calT}) \deq \set{(f, a) \in \Besov{s}{p}{q}(\T^D)
      \times \T^{d'}}{\exists x \in \calX^\alpha, (f, a, x) \in
      \widetilde{\calT}}
  \end{equation*}
  is analytic in the space $(\Besov{s}{p}{q}(\T^{D}) \times
  \T^{d'}, \Borel(\Besov{s}{p}{q}(\T^{D}) \times
  \T^{d'}))$.
  By the universal measurability Theorem~\ref{theo:unimeasurable},
  $\Pi(\widetilde{\calT})$ is then $\mu \otimes
  \Lebesgue_{d'}$-measurable.

  To conclude, we notice that $ \Phit$ can be written as
    \begin{equation*}
    \Phit : f \mapsto \int_{\T^{d'}}
    \mathbf{1}_{\Pi(\widetilde{\calT})}(f, a) \, da.
  \end{equation*}
  Since $\Pi(\widetilde{\calT})$ is  $\mu \otimes
  \Lebesgue_{d'}$-measurable, we can  apply  Fubini's theorem, so that  we conclude  that $ \Phit$ is $\mu$-measurable,  for any complete Borel measure $\mu$ on
  $\Besov{s}{p}{q}(\T^D)$. 
\end{proof}

\subsection{Probe space}  \ \mk

\label{sec:probe-space}

In  this section $q \in (0, \infty)$, with the obvious
modifications when $q = \infty$.  
%

\mk

We use the following
notation: to each $(j, \bfk) \in \Lambda^d$ we associate the unique $(J,
\bfK)$, $J \in \N$ and $\bfK \in \Z_J^d \backslash 2\Z_J^d$ satisfying $\bfK
2^{-J} = \bfk 2^{-j}$ ($\bfK
2^{-J} $ is the irreducible   version of the dyadic  point $\bfk 2^{-j}$).  Obviously, with the preceding notations, $J\leq j$.

\begin{proposition}
  \label{prop:gBesov}
 Let us define, for every  $\lambdaD = (j, (\bfk, \bfk'),(\bfl, \bfl') ) \in \Lambda^D$ 
\begin{equation}
  \label{eq:defcoefg}
  e_{\lambdaD} \deq \begin{cases} j^{-\frac{q+2}{qp}} 2^{(\frac{d}{p}-s)j}2^{-\frac{d}{p}J}   & \mbox{ if  \  $\bfl \neq 0^d$ and $\bfl' = 1^{d'}$},\\
   \ \ 0  &  \mbox{ if \  $\bfl =  0^d$ or $\bfl' \neq 1^{d'}$ }.
  \end{cases}
\end{equation}
 The function $g \deq \sum\limits_{\lambdaD \in \Lambda^D} e_{\lambdaD}
  \psi_{\lambdaD}$ belongs to $\Besov{s}{p}{q}(\T^D)$.
\end{proposition}

\begin{proof}
  Observe that $e_{\lambdaD} $ does not depend on $\bfl \in  L^d$.  Using the wavelet characterization \eqref{eq:caracbesov} of
  $\Besov{s}{p}{q}(\T^D)$, the proof boils down to studying for
  all integers $j \geq 1$ the quantity
$$
    A_j   \deq 2^{j(sp-D)} \sum_{ \glaD= (j, (\bfk, \bfk'),(\bfl, \bfl') )  : \, (\bfk, \bfk') \in \Z_j^D,  \, (\bfl, \bfl')\in \{0,1\}^D } \abs{e_{\lambdaD}}^p .
    $$
    By construction, $e_{\lambdaD}$ is the same for all       $\bfk'$ and all $\bfl$, and equals zero except when $\bfl'=1^{d'}$ and $\bfl \neq 0^d$. Thus,
        \begin{eqnarray*}
        A_{j }  &   =   & (2^d -1) 2^{j(sp-D+d')} \sum_{\bfk \in \Z_j^d  } \abs{e_{\lambdaD}}^p \leq   2^{d+j(sp - d)}    \sum_{\bfk \in \Z_j^d  }  j^{-p\frac{q+2}{qp} } 2^{(d - sp)j}
2^{- d J} \\
    & \leq  &  2^d j^{- \frac{q+2}{q } }     \sum_{\bfk \in \Z_j^d  }   2^{- d J} .
    \end{eqnarray*}
    where one should not forget that $J$ depends on
      $\bfk$. For a given integer $1\leq J\leq j$, the number of multi-integers $\bfk \in \Z_j^d  $ such that its irreducible version can be written $\bfK 2^{-J}$  (for some $\bfK$)   is exactly $2^{d(J-1)} $. Hence
      $$    A_j  \leq  2^d j^{-\frac{q+2}{q}} \sum_{J=1}^{j} 2^{d(J-1) - dJ} =     j^{-\frac2q}
$$
  which is an $l^q$ sequence.
\end{proof}

\begin{remark}
Although we did not prove it here, the singularity spectrum of $g$ (and of the functions $g^{(i)}$ below) can be explicitly computed: for every $h\in [s-d/p,s]$, $d_g(h) =  ph -ps +D$, and $d_g(h)=-\infty$ else. It is noticeable that $g$ does not enjoy the generic spectrum in $\Besov{s}{p}{q}(\T^D)$ (the generic spectrum has the same formula but the range of $h$ is $[s-D/p,s]$, not $[s-d/p,s]$). Nevertheless its traces will be shown to have the typical spectrum in $\Besov{s}{p}{q}(\T^d)$.
\end{remark}

\newcommand{\dsonde}{d_1} 
\newcommand{\lambdat}{\widetilde{\lambda}} 

Let $J_0 \geq 1$ to be fixed later and $\dsonde \deq 2^{d J_0}$.  For
each $d$-dimensional dyadic cube $\lambda  \in \Lambda^d$ at scale~$j $, we
enumerate in an arbitrary fashion $\lambda ^{(1)}, \dots,
\lambda ^{(\dsonde)}$ its  $d_1$ sub-cubes at scale $j+J_0$.

\begin{definition}
We set the
\emph{probe space} $\calP$ to be the $\dsonde$-dimensional subspace of
$\Besov{s}{p}{q}(\T^D)$ spanned by the functions $g^{(i)}$, whose
wavelet coefficients $e^{(i)}_{\lambdaD}$ are defined in the following
way: for each $\lambdaD = (j, (\bfk, \bfk'),(\bfl, \bfl') ) \in \Lambda^D$, let $\lambda \deq (j,\bfk,\bfl )$ and
\begin{equation}
\label{defprobe}
  e^{(i)}_{\lambdaD} =
  \begin{cases}
    e_{ \widetilde\gla^D} & \text{if } \lambda = \lambdat^{(i)} \text{ for some
    } \widetilde\gla^D:=( j - J_0, (\tilde\bfk,\tilde\bfk'),( \tilde\bfl, \tilde\bfl')) \\
    0 & \text{else.}
  \end{cases}
\end{equation}
\end{definition}

In the definition above, $\lambdat^{(i)}$ is the sub-cube associated with $\lambdat=( j - J_0,  \tilde\bfk ,  \tilde\bfl )$ (which is the restriction to $\zu^d$ of $\widetilde\gla^D$).

In particular, recalling \eqref{eq:defcoefg}, as soon as $\bfl'\neq 1^{d'}$, $ e^{(i)}_{\lambdaD} =0$, and this coefficient is the same for all $\bfk'\in \Z_j^{d'}$. By the same proof as Proposition~\ref{prop:gBesov}, each $g^{(i)}$
also belongs to $\Besov{s}{p}{q}(\T^D)$.

\smallskip

Heuristically, the wavelet coefficients of $g$ at generation $j$ are dispatched in wavelet coefficients at generation $j+J_0$ for the functions $g^{(i)}$, the distribution being organized so that for any cube $\glaD$, there is only one $g^{(i)}$  such that $e^{(i)}_{\lambdaD} \neq 0.$
\smallskip

Let us now consider their traces $g^{(i)}_a$ on the affine subspace
$\calH_a$. 

\begin{lemma}
For every $i\in \{1,\cdots, d_1\}$, for every $j\geq 1$, for every $\lambda=(j,\bfk,\bfl)$ with $\bfk\in \Z_j^d$ and $\bfl\neq 0^d$, we have the formula
\begin{equation}
  \label{eq:defgi}
  e^{(i)}_\lambda(a) =  e^{(i)}_{(j,(\bfk, 1^{d'}),(1^d, 1^{d'}))}  G_{d'}(2^j a)
\end{equation}
 where $G_{d'}$ was defined in \eqref{defgdprime}.
 
Moreover, if $\bfl =  0^d$, then $e^{(i)}_\lambda(a) =0$.
\end{lemma}
\begin{proof}
 Following~(\ref{eq:tracecoef4}) and ~(\ref{eq:tracecoef3}), the wavelet coefficients
of these traces are: for all $j\geq 1$, for all $\gla = (j,\bfk,\bfl) \in \{j\}\times \Z_j^d\times \{0,1\}^d$,  
\begin{eqnarray}
 \label{eq:tracecoef10}  \mbox{if $\bfl=0^d$, } &   e^{(i)}_\lambda(a)    \deq  & \sum_{ \substack{ \glaD =(j,(\bfk,\bfk'),(0^d,\bfl')): \\   \bfk' \in \Z_j^{d'},  \,  \bfl' \in L^{d'}}}
    e^{(i)}_{\lambdaD} \prodprimePsi.\\
 \label{eq:tracecoef11} \mbox{if $\bfl  \in L^d$,  } &   e^{(i)}_\lambda(a)    \deq  & \sum_{ \substack{\glaD =(j,(\bfk,\bfk'),(\bfl,\bfl')): \\   \,   \,\bfk' \in \Z_j^{d'},  \, \bfl' \in \{0, 1\}^{d'}}}
    e^{(i)}_{\lambdaD} \prodprimePsi.
\end{eqnarray}
By definition  of $e^{(i)}_{\lambdaD}$, the coefficients \eqref{eq:tracecoef10} are all zero. Now,  remember that by construction $e^{(i)}_{\lambdaD}$ does not depend on
$\bfk'$, nor on $\bfl\in L^d$, and that they all have the same values as one of them, say the one  with  $\bfl=1^d$. Thus, as soon as $\bfl\neq 0^d$, formula \eqref{eq:tracecoef11} can be simplified into
\begin{eqnarray*}
  e^{(i)}_\lambda(a)  & = &  e^{(i)}_{(j,  \bfk,  1^d )} (a) =  \sum_{  {\glaD =(j,(\bfk,\bfk'),(1^d, 1^{d'}):    \,   \,\bfk' \in \Z_j^{d'} }}  e^{(i)}_{\lambdaD} \ \prod_{i=1}^{d'} \Psi_{j,k'_i}^1 (a_i) \\
  & = & e^{(i)}_{(j,(\bfk, 1^{d'}),(1^d, 1^{d'}))}   \sum_{     \,\bfk' \in \Z_j^{d'}  }   \ \prod_{i=1}^{d'} \Psi^1 (2^ja_i -k'_i).
\end{eqnarray*}
Since $\Psi^1$ has compact support, for a given $a\in (0,1)^{d'}$, when $j$ is large enough, we have 
$$ \sum_{   \bfk' \in \Z_j^{d'}  }   \ \prod_{i=1}^{d'} \Psi^1 (2^ja_i -k'_i) = \sum_{   \bfk' \in \Z ^{d'}  }   \ \prod_{i=1}^{d'} \Psi^1 (2^ja_i -k'_i) = G_{d'}(2^j a).$$
 This yields \eqref{eq:defgi}.
\end{proof}

\label{sec:probe-space-1}

\mk

Let $x \in \calX^\alpha$ not a dyadic element of $\zu^d$, and consider the irreducible  sequence $(J_n, \bfK_n)$
associated to $x$ as in (\ref{eq:eq311}), i.e. 
$$|x - \bfK_n2^{-J_n}| \leq 2^{-\alpha J_n}.$$
  Let $a \in \Aun$ and let
$j_a$ be the associated integer constructed  in Proposition~\ref{prop:protr1}.
Let $n$ be such that $j_n \deq [{\alpha J_n}] \geq j_a$ and such that \eqref{eq:defgi} holds. Let us
denote by $\lambda_n \deq (j_n, \bfk_n, \bfl)$ the unique  dyadic node (unique in the sense that $\bfl$ varies in $L^d$) such that
$\bfK_n 2^{-J_n} = \bfk_n 2^{-j_n}$. With each $\gla_n$  can be associated its sub-cubes $\gla_n^{(i)}$, $i\in \{1,\cdots, d_1\}$. 
 
\begin{lemma}
  \label{lemm:lemtr2}
  For all $1 \leq i \leq \dsonde$, $\lambda_n^{(i)}$ lies within the cone
  of influence of width $2^{J_0+2}$ above $x$, and
  \begin{equation}
    \label{eq:eq312}
    \abs{e^{(i)}_{\lambda_n^{(i)}}(a)} \geq  C 
    j_n^{-(2d'+\frac{q+2}{qp})} 2^{-H(\alpha) j_n}
  \end{equation}
  the constant $C$ depending only on $J_0$.
\end{lemma}
\begin{proof}
 Remark that   $\lambda_n^{(i)}  $ can be written $\lambda_n^{(i)}= (j_n + J_0, \bfk_n^{(i)}, \bfl)$ for some integer $\bfk_n^{(i)} \in \Z_{j_n+J_0}^d$.  By construction, 
 $$ | \bfk_n^{(i)} 2^{- (j_n + J_0)} - \bfK_n2^{-J_n}|  = | \bfk_n^{(i)} 2^{- (j_n + J_0)} - \bfk_n2^{-j_n}|  \leq 2^{-j_n}.$$
  Using~(\ref{eq:eq311}) we deduce that 
\begin{eqnarray*}
 \big|x - \bfk_n^{(i)} 2^{- (j_n + J_0)} \big|  &  \leq  &   \big|{x - \ \bfK_n2^{-J_n} } \big| +  \big| \bfk_n^{(i)} 2^{- (j_n + J_0)} - \bfK_n2^{-J_n} \big|  \\
  &  \leq  &    {2^{-\alpha
      J_n}} + {2^{-j_n}} \leq 3 {2^{-j_n}} \leq  ( { 2^{J_0+2}}){2^{- (j_n+J_0)}}.
\end{eqnarray*}
This shows the first part of Lemma \ref{lemm:lemtr2}.

\mk

Recall now Proposition \ref{prop:protr1}. The fact that 
  $a \in \calA_1$   guarantees that
$|G _{d'}(2^ja)| \geq j^{-2d'}$ as soon as $j\geq j_a$.  Combining this 
with~(\ref{eq:defgi}), we get
$$
 \abs{ e^{(i)}_{(\lambda_n)^{(i)}}(a) }  \geq  (j_n+J_0)^{-2d'}
  \abs{  e^{(i)}_{(j_n+J_0,(\bfk_n^{(i)}, 1^{d'}),(1^d, 1^{d'}))}  }.$$
 Remembering now how we chose the coefficients of $g^{(i)}$ in \eqref  {defprobe}, we see that   
\begin{eqnarray*}  
 \abs{ e^{(i)}_{(\lambda_n)^{(i)}}(a) }   &\geq  &  (j_n+J_0)^{-2d'}  \cdot \abs{  e _{(j_n ,(\bfk_n, 1^{d'}),(1^d, 1^{d'}))}  } \\ 
&\geq  &  (j_n+J_0)^{-2d'} \cdot   (j_n  )^{-\frac{q+2}{qp}}   2^{( \frac{d}{p}-s) j_n -
    \frac{d}{p} J_n} \\ 
  &\geq  &  C j_n^{-(2d'+\frac{q+2}{qp})} 2^{-H(\alpha) j_n} 
\end{eqnarray*}
where we used that $\alpha J_n\leq j_n+1$.
\end{proof}

\subsection{Shyness of $\Fagn$}\ 

\

Recall that $\gamma >H(\alpha)$.

Take an arbitrary $f \in \Besov{s}{p}{q}(\T^D)$ with wavelet coefficients $c_\glaD$, and for
each $\beta \in \R^{\dsonde}$ define
\begin{equation*}
  f^\beta \deq f + \sum_{i=1}^{\dsonde} \beta_i g^{(i)}.
\end{equation*}
As usual now, $f^\beta_a$ will denote its trace at level $x' = a$ and
$c^\beta_\lambda(a)$ the wavelet coefficients of that trace.
Now we choose $J_0$ large enough so that 
\begin{equation}
\label{choixj0}
 d - \dsonde (\gamma -
H(\alpha)) < 0
\end{equation}
Our goal is to prove:
\begin{proposition}
  \label{prop:timide}
 For any $f \in \Besov{s}{p}{q}(\T^D)$, the set $\set{\beta \in \R^{d_1}}{f^\beta \in \Fagn}$ has
  $d_1$-dimensional Lebesgue  measure $\mathcal{L}_{d_1}$ equal to 0.
\end{proposition}
This   will show that $\Fagn$ is shy. Let us quickly explain this fact.

Let us denote by $\mu$   the measure $\mathcal{L}_{d_1}$  carried by $\calP$.  Assume that  Proposition \ref{prop:timide} holds true, and fix any  $f \in \Besov{s}{p}{q}(\T^D)$. For $\mu$-almost $F \in \calP$, we know that $f+F \notin  \Fagn$.  Hence the set    $\{ f+ \Fagn\} \cap \calP$  has a $\mu$-measure equal to $0$, i.e. 
$$\mu(\{ f+ \Fagn\} ) =0.$$
Since this is true for any $f \in \Besov{s}{p}{q}(\T^D)$, by Definition \ref{defprevalence}, the set $\Fagn$ is shy.

\mk

Before that, two intermediary lemmas are necessary. Let us introduce
\begin{equation*}
  \calB_a \deq \set{  \beta\in
    \mathbb{R}^{\dsonde}}{
    \begin{array}{c}
      \exists  \, x_\beta \in \calX^{\alpha},  \ \forall \,  \lambda \in
      \Lambda^d, \\ 
      \abs{c^\beta_\lambda(a)} \leq N 2^{-\gamma
        j} ( 1 + \abs{2^j x_\beta - \bfk} )^\gamma 
    \end{array} 
  }
\end{equation*}

\begin{lemma}
  \label{lemm:protr14}
  The application  $(a, \beta) \mapsto \mathbf{1}_{\calB_a}(\beta)$
  is Lebesgue-measurable  as an application from $\T^{d'}\times \R^{\dsonde}$ to $\R$.
\end{lemma}

\newcommand{\phip}{\phi}
\begin{proof}
  Let $\phip: (a, \beta, x) \mapsto \inf\limits_{\lambda \in
    \Lambda^d} N2^{-\gamma j}(1 + \abs{2^jx -\textbf{k}})^\gamma -
  \abs{c^\beta_{\lambda }(a)}$. By an argument similar to the one used in proving Proposition \ref{prop:mesure},   $\phip$ is Borel on $\T^{d'}\times \R^{\dsonde} \times
  \T^d$.
  Remark then that $\mathbf{1}_{\calB_a}(\beta) $ can be written as $\mathbf{1}_{\calB_a}(\beta)= \mathbf{1}_{\calG}(a,
  \beta)$, where
  \begin{equation*}
    \calG \deq \set{ (a, \beta) \in \T^{{d'}} \times
      \mathbb{R}^{\dsonde}}{\exists  x \in  \calX^{\alpha}, 
      \phip(a, \beta, x) \geq 0 }.
  \end{equation*}
 This set  can be written as 
  \begin{equation}
  \label{eq22}
 \pi\left(   \phip^{-1}([0, \infty)) \bigcap  
    \paren{ \T^{d'} \times \mathbb{R}^{\dsonde} \times  \calX^{\alpha} }  \right),
  \end{equation}
  where $\pi (a,\beta,x)= (a,\beta)$ is the (continuous) canonical projection on the two first  coordinates. 
  Since the set between brackets in \eqref{eq22}  is clearly a Borel set,  $\calG$ is analytic and in particular,
  by Theorem~\ref{theo:unimeasurable}, it is Lebesgue-measurable. By a Fubini argument, we deduce Lemma \ref{lemm:protr14}.
\end{proof}

\begin{lemma}
  \label{lemm:BaaNf}
  For each $a \in \Aun$, the set $\calB_a$ has Lebesgue measure 0.
\end{lemma}

\begin{proof}
  For any $\lambda_0 \deq (j_0, \bfk_0,\bfl_0) \in \Lambda^d$ we put
  \begin{displaymath}
    \calB_{a, \lambda_0} \deq \set{ \beta\in \mathbb{R}^{\dsonde}}{
       \begin{array}{c} 
        \exists \,  x_\beta \in B(\bfk_0 2^{-j_0}, 2^{ -  \alpha j_0}),
        \forall \lambda \in \Lambda^d, \\ 
        \abs{c^\beta_{\lambda }(a)} \leq N
        2^{-\gamma j} (1 + \abs{ 2^j x_\beta - \bfk } )^\gamma 
      \end{array}
    }
  \end{displaymath} 
  so that
  \begin{displaymath}
    \calB_a = \limsup_{j_0 \to \infty} \bigcup_{k_0 \in \Z_{j_0}^d} 
    \calB_{a, \lambda_0}  = \bigcap_{j\geq 1}  \ \bigcup_{j_0\geq j}  \ \bigcup_{\bfk_0 \in \Z_{j_0}^d}  \ \calB_{a, \lambda_0} 
  \end{displaymath}

  \newcommand{\betat}{\widetilde{\beta}}

  We want to show that $\Lebesgue_{\dsonde}(\calB_a) = 0$ by bounding by above
  each $\Lebesgue_{\dsonde}(\calB_{a, \lambda_0})$.  Suppose that $\beta$
  and $\betat$ both belong to some $\calB_{a, \lambda_0}$, where $j_0 $ is
  large enough so that $j_1 \deq \floor{\alpha j_0} \geq j_a$
  (cf. Proposition \ref{prop:protr1}).

  Applying Lemma \ref{lemm:lemtr2}, there exists $\lambda_1= (j_1,
  \bfk_1,\bfl_1)$ such that for all $1 \leq i \leq \dsonde$,
  \begin{enumerate}
   
  \item \label{item:1} $\lambda_1^{(i)}$ is in the cone of
    influence of width $2^{J_0+1}$ above $x_\beta$ and $x_{\betat}$,
	
  \item \label{item:2} $\abs{e^{(i)}_{\lambda_1^{(i)} }(a)} \geq  C
    j_1^{-(q_{d'}+\frac{q+2}{qp})} 2^{-H(\alpha)j_1}$, 
	

  \end{enumerate}

  From~(\ref{item:1}) we deduce that
  \begin{equation*}
    \biggabs{c_{\lambda_1^{(i)} }^\beta(a)} =
    \biggabs{c_{\lambda_1^{(i)} }(a) + \sum_{i=1}^{\dsonde} \beta_i  
      e^{(i)}_{\lambda_1^{(i)} }(a)} \leq \frac{C}{2} 2^{-\gamma j_1} 
  \end{equation*}
  ($C$ depending on $N, J_0, \gamma$) and the same for $\betat$.
  On the other hand, by construction of the functions $g^{(i)}$, we have
  $e^{(i')}_{\lambda_1^{(i)} }(a) = 0$ for any $i \neq i'$.  Thus
  \begin{align*}
    \abs{(\beta_{i }-\betat_{i })e^{i}_{\lambda_1^{(i)}
      }(a) } 
    &= \left| \biggparen{ c_{\lambda_1^{(i)}    }(a)+\sum^{\dsonde}_{i=1}\beta_i
        e^{i}_{\lambda_1^{(i)} }(a) } \right. \\
    &  \hspace{10em} \left. - \biggparen{
        c_{\lambda_1^{(i)}
        }(a)+\sum^{\dsonde}_{i=1}\betat_{i}
        e^{i}_{\lambda_1^{(i)}  }(a) } \right| \\ 
    &\leq \abs{c_{\lambda_1^{(i)}    }(a)+\sum^{\dsonde}_{i=1}\beta_i
      e^{i}_{\lambda_1^{(i)} }(a)} + \abs{c_{\lambda_1^{(i)}
      }(a)+\sum^{\dsonde}_{i=1}\betat_{i}
      e^{i}_{\lambda_1^{(i)}  }(a) } \\ 
    &\leq C 2^{-\gamma j_1}.
  \end{align*}
  
\noindent Recall that $\gamma >H(\alpha)$. Combining this with~(\ref{item:2}), we deduce that
  \begin{align*}
    \abs{\beta_{i }-\betat_{i}} 
    &\leq C 2^{-(\gamma-H(\alpha))j_1} j_1^{q_{d'}+\frac{q+2}{qp}}  
  \end{align*}
  hence
  \begin{displaymath}
    \Lebesgue_{\dsonde} ( \calB_{a, \lambda_0}   ) \leq C
    2^{-\dsonde (\gamma - H(\alpha)) j_1} j_1^{\dsonde(q_{d'}+\frac{q+2}{qp})}
  \end{displaymath}
  Summing over all $2^{d j_0}$ nodes $\lambda_0$ at scale $j_0$ we
  conclude that
  \begin{align*}
    \Lebesgue_{\dsonde} \biggparen{ \bigcup_{k_0 \in \Z_{j_0}^d}
      \calB_{a, \lambda_0} } 
    &\leq C  2^{d j_0 - \dsonde (\gamma - H(\alpha)) j_1}
    j_1^{(q_{d'}+\frac{q+2}{qp})\dsonde} \\
    & \leq C 2^{j_0( d -\dsonde\alpha(\gamma-H(\alpha) ) )} (\alpha
    j_0)^{(q_{d'}+\frac{q+2}{qp})\dsonde}
  \end{align*}
  whose series converges because of our choice \eqref{choixj0} for $J_0$.     The Borel-Cantelli lemma implies that
  $\Lebesgue_{\dsonde}(\calB_a) = 0$.
\end{proof}

\begin{proof}[Proof of Proposition~\ref{prop:timide}]
  We can rewrite the result of Lemma~\ref{lemm:BaaNf} as
  \begin{equation*}
    \int_{\T^{d'}} \int_{\R^{\dsonde}}
    \mathbf{1}_{\calB_a}(\beta) \,  d\beta  \, da = 0 
  \end{equation*}
  Applying Fubini's theorem (measurability being guaranteed by
  Lemma~\ref{lemm:protr14}),
  \begin{equation*}
    \int_{\R^{\dsonde}} \int_{\T^{d'}}
    \mathbf{1}_{\calB_a}(\beta) \,  da \, d\beta = 0 
  \end{equation*}
  In other words, for almost all $\beta \in \R^{\dsonde}$, there
  exists a set $\calA_\beta(f)$ of full Lebesgue measure in $\T^{d'}$
  such that $a \in \calA_\beta(f)$ implies $\beta \not \in \calB_a$.
  This in turn implies $f^\beta \not\in \Fagn$ and the announced
  result follows by complementarity in~$\R^{\dsonde}$.
\end{proof}

Our proof is now complete.

\appendix

\section{Direct proof of Theorem~\ref{theo:tracebesov}}
\label{sec:proof-coroll-refc}

\newcommand{\indexk}{\bfk \in \Z_j^d}
\newcommand{\indexkp}{\bfk' \in \Z_j^{d'}}

It suffices to show that for any $\eps > 0$, the set 
$$
  \set{a \in \zu^{d'}}{f_a \in \Besov{s-\eps}{p}{\infty}(\T^d)} 
$$
has full Lebesgue measure in $\T^{d'}$.  Then we will get the
announced result by considering a sequence $\eps_n$ decreasing to 0.

Recall that $f_a=F_a+G_a$, where $F_a$ and $G_a$ are defined in \eqref{eq:tracecoef} and \eqref{eq:tracecoef2}.   Recall also  that the wavelet coefficients of $f_a$ are denoted by $d_\gla(a)$ in its wavelet decomposition \eqref{decompffinal}, while those of $f$ are denoted by $c_\glaD$.

First, we are going to apply Proposition \ref{proptrace1}. Recalling the definition  \eqref{eq:tracecoef4} of the wavelet coefficients of $G_a$, we need to bound by above the sum
 \begin{align}
 \nonumber
  \sum_{ \gla \in \Lambda_j^d\times 0^d} \abs{d_\lambda(a)}^p &= \sum_{\gla \in \Lambda_j^d\times 0^d}
  \biggabs{\sum_{\indexkp,\bfl'\in L^{d'}} c_{\lambdaD} \prodprimePsi }^p \\
  \nonumber&\leq C \sum_{\glaD\in \Lambda_j^D\times \{0^d\times L^{d'} \}} \abs{c_{\lambdaD}}^p \biggabs{
    \prodprimePsi }^p\\
 &\leq C \sum_{\glaD\in \Lambda_j^D\times \{0^d\times \{0,1\}^{d'} \} } \abs{c_{\lambdaD}}^p \biggabs{
    \prodprimePsi }^p \label{eqfin1}
  \end{align}
for some constant $C$ that depends only on $\Psi^0$ and $\Psi^1$.   Indeed, since
$\Psi^0$ and $\Psi^1$  are compactly supported, for each $a$ only a
finite number (independent of $j$ or $a$) of terms in the second sum
are non zero. 

Now, using  definition \eqref{eq:tracecoef3}  for the wavelet coefficients of $F_a$, we find that  
\begin{align}
 \nonumber  \sum_{ \gla \in \Lambda_j^d\times L^d} \abs{d_\lambda(a)}^p &= \sum_{\gla \in \Lambda_j^d\times L^d}
  \biggabs{\sum_{\indexkp,\bfl'\in \{0,1\}^{d'}} c_{\lambdaD} \prodprimePsi }^p \\
  &\leq C \sum_{\glaD\in \Lambda_j^D\times   \{L^d\times  \{0,1\}^{d'}\}} \abs{c_{\lambdaD}}^p \biggabs{
    \prodprimePsi }^p\label{eqfin2}
\end{align}
again for some constant $C$ that depends only on $\Psi^0$ and $\Psi^1$. 

We now consider the sum of all wavelet coefficients of $f_a$, and we  integrate it over $a \in \T^{d'}$. Using \eqref{eqfin1} and \eqref{eqfin2}, and recalling that $L^d\cup 0^d = \{0,1\}^d$,  we get
$$
  \int_{\T^{d'}} \sum_{\gla \in \Lambda_j^d\times \{0,1\}^d} \abs{ d_{\lambda }(a) }^p da \leq C
  \sum_{\glaD\in \Lambda_j^D\times \{0,1\}^D} \abs{c _{\lambdaD} }^p \int_{\T^{d'}}
  \biggabs{\prodprimePsi }^p da.$$
 The attentive reader has noticed that there is no wavelet coefficient  associated with the index $l^D=0^D$ for the function $f$, hence the sum of $\glaD$ over $\Lambda_j^D\times \{0,1\}^D$ is the same as the sum of $\glaD $ over $\Lambda_j^D\times L^D$.
 
  Since the functions $\Psi_{j,k'_i}^{l'_i}$   are  bounded uniformly in $j$,  $\bfk'$ and $\bfl'$, and has support
    width $\leq K 2^{-j}$,
    $$  \int_{\T^{d'}} \sum_{\gla \in \Lambda_j^d \times \{0,1\}^d} \abs{ d_{\lambda }(a) }^p da  \leq C \sum_{\glaD\in \Lambda_j^D\times L^D}
  \abs{c_{\lambdaD} }^p 2^{-j d'} .$$

  Using the definition of Besov
    norm~(\ref{eq:caracbesov}), we see that
      $$  \int_{\T^{d'}} \sum_{\gla \in \Lambda_j^d \times \{0,1\}^d} \abs{ d_{\lambda }(a) }^p da  \leq C \norm{\Besov{s}{p}{\infty}}{f}
  2^{(d-sp) j}
$$
for some other constant $C$ that still depends only on $\Psi$.

Let us now define
\begin{equation*}
  \calA_{j} \deq \biggset{a \in \T^{d'}}{\sum_{\gla \in \Lambda_j^d \times \{0,1\}^d} \abs{
      d_{\lambda }(a) }^p > C  \norm{\Besov{s}{p}{\infty} (\zu^D) }{f} 2^{-((s-\eps)p-d)j} }
\end{equation*}
By  Markov's inequality, it follows that
\begin{equation*}
  \Lebesgue_{d'}(\calA_{j}) 
  \leq \frac{C \norm{\Besov{s}{p}{\infty} (\zu^D)}{f} 2^{(d-sp) j}}{C
    \norm{\Besov{s}{p}{\infty} (\zu^D)}{f} 2^{-((s-\eps)p-d)j} } = 2^{-j \eps p}
\end{equation*}
thus, applying the Borel-Cantelli lemma, $\Lebesgue_{d'}(\limsup_j
\calA_j) = 0$.  

Finally, by construction,  for any $a \in \T^{d'} \backslash
\limsup_j \calA_j$, there exists $j_0$ such that $j \geq j_0$ implies
\begin{equation*}
  \sum_{\gla \in \Lambda_j^d \times \{0,1\}^d}
  \abs{ d_{\lambda }(a) }^p 2^{(sp-d)j} \leq C
  \norm{\Besov{s}{p}{\infty} (\zu^D)}{f},
\end{equation*}
hence $f_a \in \Besov{s}{p}{\infty}(\T^d)$. \hfill \qed

\section{Proof of Proposition~\ref{prop:mesure}, case $q = \infty$}
\label{sec:proof-key-nonsep}

The only serious difference is due to the fact that
$\Besov{s}{p}{\infty}(\T^D)$ is no longer separable, so the argument
for universal measurability of the ancillary set
\begin{equation*}
  \Fagn \deq \set{f\in \Besov{s}{p}{\infty}(\T^D)}{\Lebesgue_{d'}(\Aagn) > 0}
\end{equation*}
 has to use a different definition of analyticity.

\subsection{Analytic sets in non-Polish spaces}  \ \mk

\label{sec:analytic-sets}

\newcommand{\stimes}{\times}

Analytic sets were previously defined in Polish spaces as continuous
images of Borel sets: this cannot apply to
$\Besov{s}{p}{\infty}(\T^D)$.  However we can use the following more
general definition, adapted from~\cite{Choquet:1954uo}, for any
Hausdorff topological space $X$ endowed with its Borel
$\sigma$-algebra $\Borel(X)$.  First, for a compact topological space
$K$ we write $\calK$ the collection of its closed subsets and
$(\Borel(X) \stimes \calK)_{\sigma \delta}$ the collection of
countable intersection of countable unions of sets that are Cartesian
products of a Borel set in $X$ and a closed set in $K$ and $\pi: X
\times K\rightarrow X$ is the canonical projection map 
$\pi(x, y) = x$.

\smallskip

\begin{definition}
  \label{defi:analnonsep}
  A set $A \subset X$ is said to be \emph{analytic} if there exists a
  compact space $K$ and $S \in (\Borel(X) \stimes \calK)_{\sigma
    \delta}$ such that
  \begin{equation*}
    A = \pi(S).
  \end{equation*}
\end{definition}

It is easy to check that this definition coincides with the previous
one when $X$ is Polish.  Furthermore, in this framework,
Theorem~\ref{theo:unimeasurable} (based on Choquet's capacitability
theorem) now holds in any Hausdorff topological space (see
\cite{Choquet:1954uo}).  

\subsection{Measurability}  \ \mk

\label{sec:measurability-nonsep}

\begin{proposition}
  \label{prop:mesure-nonsep}
  The set $\Fagn$ is universally measurable in
  $\Besov{s}{p}{\infty}(\T^D)$.
\end{proposition}

\newcommand{\dimkp}{{\Delta_j}}

\begin{proof}
  We use the same notations $\Phi_\lambda$, $\Phi$,
  $\widetilde{\calT}$ and $\Pi$ as in the proof of
  Proposition~\ref{prop:mesure}.  We will show that the set
  $\Pi(\widetilde{\calT})$ is analytic in the sense of
  Definition~\ref{defi:analnonsep}, with $X \deq
  \Besov{s}{p}{\infty}(\T^D) \times \T^{d'}$ and $K \deq \T^d$.  For
  short let us put $\dimkp \deq \Z_j^{d'} \times \{0, 1\}^{d'}$.
  Given $n \in \N$, $m \in \Z^{\dimkp}$ and $m' \in \Z^{d'}$ we write
  \begin{align*}
    Q(n, m) &\deq 2^{-n} \prod_{(k',l') \in \dimkp} [m_{k',l'},
    m_{k',l'} + 1] \\
    Q'(n, m') &\deq \ 2^{-n} \prod_{1\leq i \leq d'} \  [m'_i, m'_i + 1] \\
    Q(n, m, m') &\deq Q(n, m) \times Q'(n, m')
  \end{align*}

  Having fixed $\lambda = (j, \bfk, \bfl) \in \Lambda^d \times L^d$
  and considering $\lambdaD =(j, (\bfk, \bfk'), (\bfl, \bfl'))$, any
  $f \in \Besov{s}{p}{\infty}(\T^D)$ induces a map $s_\lambda :
  (\bfk',\bfl') \mapsto c_{\lambdaD}$ that we identify to an element
  of $\R^\dimkp$. Then we define
  \begin{equation*}
    \Theta(s_\lambda, a, x) \deq N 2^{-\gamma j} (1
      + \smallabs{2^jx - \bfk})^\gamma - \abs{c_{\lambda}(a)} 
  \end{equation*}
  as well as
  \begin{equation*}
    X_\lambda(n, m, m') \deq \biggset{x \in
      \T^d}{\sup_{(s_\lambda, a) \in Q(n, m, m')} \Theta(s_\lambda, a,
      x) \geq 0}
  \end{equation*}
  The dependency of $c_\lambda(a)$ on $s_\lambda$ and $a$ is given
  in \eqref{eq:tracecoef4}, \eqref{eq:tracecoef3} and \eqref{decompffinal2} and it is continuous. So is the function
  $\Theta$.  Since $Q(n, m, m')$ is compact, it follows that
  $X_\lambda(n, m, m')$ is closed (Lemma~\ref{lemm:compactsup}).
  Furthermore, if we put
  \begin{equation*}
    F_\lambda(n, m) \deq \set{f \in
      \Besov{s}{p}{\infty}(\T^D)}{s_{\lambda} \in Q(n, m)}
  \end{equation*}
  then it is clear by continuity of $\Phi_\lambda$ that
  \begin{equation*}
    \Phi_\lambda^{-1}([0, \infty)) = \bigcap_{n \in \N}
    \bigcup_{(m,m') \in \Z^{\dimkp} \times \Z_j^{d'}} F_\lambda(m,n)
    \times Q'(m', n) \times   X_\lambda(m, m', n) 
  \end{equation*}
  This proves that $\Phi_\lambda^{-1}([0,\infty)) \in (\Borel(X)
  \stimes \calK)_{\sigma \delta}$.  We deduce that $\Phi^{-1}([0,
  \infty)) = \bigcap_\lambda \Phi_\lambda^{-1}([0,\infty))$ and the
  set $\widetilde{\calT}$ belong to $(\Borel(X) \stimes \calK)_{\sigma
    \delta}$ as well because $\Besov{s}{p}{\infty}(\T^D) \times \T^{d'}
  \times \calX^\alpha$ is obviously in $(\Borel(X) \stimes
  \calK)_{\sigma \delta}$.  Its projection $\Pi(\widetilde{\calT})$ is
  thus analytic and we conclude in the same way as for
  Proposition~\ref{prop:mesure}.
\end{proof}

\begin{lemma}
  \label{lemm:compactsup}
  Let $A$ and $B$ be topological spaces, $A$ compact and $B$ locally
  compact.  If $f$ is continuous: $A \times B \to \R$, then $f_s : b
  \mapsto \sup_{a \in A} f(a, b)$ is continuous on $B$.
\end{lemma}

\begin{proof}
  Recall that a Hausdorff space-valued function defined on a compact
  set is continuous if and only if its graph is compact.  Continuity
  being a local property, we can suppose without loss of generality
  that $B$ is also compact.  The graph $\Gamma$ of $f$ is then compact
  and so is its image by the projection $\varpi : (a, b, y) \mapsto
  (b, y)$.  As a supremum of continuous functions, $f_s$ is lower
  semi-continuous, so its epigraph $E$ is closed.  But the graph of
  $f_s$ is precisely $E \cap \varpi(\Gamma)$, so it is compact; it
  follows that $f_s$ is continuous.
\end{proof}

\mk\mk

\subsection*{Acknowledgments}
 
The authors  are thankful to  Basarab Matei for his  simulations and pictures of Daubechies' wavelets.

\bibliographystyle{acm}

\end{document}